\documentclass[]{amsart}

\usepackage{cite}

\makeatletter
\def\@tocline#1#2#3#4#5#6#7{\relax
  \ifnum #1>\c@tocdepth 
  \else
    \par \addpenalty\@secpenalty\addvspace{#2}%
    \begingroup \hyphenpenalty\@M
    \@ifempty{#4}{%
      \@tempdima\csname r@tocindent\number#1\endcsname\relax
    }{%
      \@tempdima#4\relax
    }%
    \parindent\z@ \leftskip#3\relax \advance\leftskip\@tempdima\relax
    \rightskip\@pnumwidth plus4em \parfillskip-\@pnumwidth
    #5\leavevmode\hskip-\@tempdima
      \ifcase #1
       \or\or \hskip 1em \or \hskip 2em \else \hskip 3em \fi%
      #6\nobreak\relax
    \hfill\hbox to\@pnumwidth{\@tocpagenum{#7}}\par
    \nobreak
    \endgroup
  \fi}
\makeatother

\usepackage{amsmath}
\usepackage{amsfonts}
\usepackage{amssymb}
\usepackage{amsthm}
\usepackage{enumitem}
\usepackage{bm}
\usepackage{hyperref}
\usepackage{tikz}
\usetikzlibrary{positioning}
\usetikzlibrary{arrows}
\usepackage{tkz-berge}
\usepackage{tkz-graph}
\usepackage{tikz-cd}

\usepackage{rotating}
\usepackage{color}

\newtheorem{theorem}{Theorem}
\newtheorem{prop}[theorem]{Proposition}

\newtheorem{corollary}[theorem]{Corollary}

\newtheorem{lemma}[theorem]{Lemma}
\newtheorem{conjecture}[theorem]{Conjecture}
\newtheorem{problem}[theorem]{Problem}
\theoremstyle{definition}
\newtheorem{definition}[theorem]{Definition}
\newtheorem{example}[theorem]{Example}

\theoremstyle{remark}
\newtheorem{remark}[theorem]{Remark}
\numberwithin{theorem}{section}
\numberwithin{equation}{section}

\theoremstyle{citeplain}
\newtheorem{citethm}[theorem]{Theorem}

\DeclareMathOperator{\rk}{rk}

\DeclareMathOperator{\Span}{Span}

\DeclareMathOperator{\Hom}{Hom}
\DeclareMathOperator{\Char}{Char}
\DeclareMathOperator{\Maps}{Maps}
\DeclareMathOperator{\conv}{conv}
\DeclareMathOperator{\Hilb}{Hilb}
\DeclareMathOperator{\Cone}{Cone}

\DeclareMathOperator{\Sym}{Sym}
\DeclareMathOperator{\Spec}{Spec}

\def\AA{\mathbb{A}}
\def\CC{\mathbb{C}}
\def\QQ{\mathbb{Q}}
\def\RR{\mathbb{R}}
\def\ZZ{\mathbb{Z}}
\def\PP{\mathbb{P}}

\def\SS{\mathbb{S}}
\def\Z+{\mathbb{Z}_{\geq 0}}
\def\R+{\mathbb{R}_{\geq 0}}
\def\F{\CC}

\def\bk{\mathbf{k}}
\def\bu{\mathbf{u}}
\def\bv{\mathbf{v}}
\def\ba{\mathbf{a}}
\def\bt{\mathbf{t}}
\def\bb{\mathbf{b}}
\def\be{\mathbf{e}}
\def\pM{\mathcal{M}}
\def\pN{\mathcal{N}}

\newcommand\restr[2]{{
		\left.\kern-\nulldelimiterspace 
		#1 
		\vphantom{\big|} 
		\right|_{#2} 
}}

\def\K{\mathbb{C}}
\def\F{\mathbb{F}}

\title{Flag matroids: algebra and geometry}
\author{Amanda Cameron, Rodica Dinu, Mateusz Micha{\l}ek and Tim Seynnaeve}

\address{AMANDA CAMERON, MAX PLANCK INSTITUTE FOR MATHEMATICS IN THE SCIENCES, INSELSTR. 22, 04103, LEIPZIG, GERMANY, AND DEPARTMENT OF MATHEMATICS AND COMPUTER SCIENCE, EINDHOVEN UNIVERSITY OF TECHNOLOGY, PO BOX 513, 5600 MB EINDHOVEN, NETHERLANDS}
\email{a.r.cameron@tue.nl}

\address{RODICA DINU, FACULTY OF MATHEMATICS AND COMPUTER SCIENCE, UNIVERSITY OF BUCHAREST, STR. ACADEMIEI 14, 010014 BUCHAREST, ROMANIA}
\email{rdinu@fmi.unibuc.ro}

\address{MATEUSZ MICHA{\L}EK, MAX PLANCK INSTITUTE FOR MATHEMATICS IN THE SCIENCES, INSELSTR. 22, 04103, LEIPZIG, GERMANY, AND INSTITUTE OF MATHEMATICS OF THE POLISH ACADEMY OF SCIENCES, UL. \`{S}NIADECKICH 8, 00-656 WARSZAWA, POLAND}
\email{Mateusz.Michalek@mis.mpg.de}

\address{TIM SEYNNAEVE, MAX PLANCK INSTITUTE FOR MATHEMATICS IN THE SCIENCES, INSELSTR. 22, 04103, LEIPZIG, GERMANY}
\email{Tim.Seynnaeve@mis.mpg.de}
\thanks{ The first author was supported by the ERC consolidator grant 617951. Most of the work in this paper was done while the author was employed by the Max Planck Institute. This research was performed while the second author was visiting Max Planck Institute, May-June and September, 2018. The third author was supported by the Polish National
Science  Centre  grant  no.~2015/19/D/ST1/01180}
\subjclass[2010]{05B35, 52B40, 14M15, 14M25, 19E08}
\keywords{Flag matroids, Tutte polynomial, torus orbits, Grassmannians, K-theory of flag varieties }
\begin{document}
\begin{abstract}
Matroids are ubiquitous in modern combinatorics. As discovered by Gelfand, Goresky, MacPherson and Serganova there is a beautiful connection between matroid theory and the geometry of Grassmannians: realizable matroids correspond to torus orbits in Grassmannians. Further, as observed by Fink and Speyer general matroids correspond to classes in the $K$-theory of Grassmannians. This yields in particular a geometric description of the Tutte polynomial.
In this review we describe all these constructions in detail, and moreover we generalise some of them to polymatroids. More precisely, we study the class of flag matroids and their relations to flag varieties. In this way, we obtain an analogue of the Tutte polynomial for flag matroids. 
\end{abstract}

\maketitle
\tableofcontents
\section{Introduction}
The aim of this article is to present beautiful interactions among matroids and algebraic varieties. Apart from discussing classical results, we focus on a special class of polymatroids known
as \emph{flag matroids}. The ultimate result is a definition of a \emph{Tutte polynomial} for flag matroids. Our construction is geometric in nature and follows the ideas of Fink and Speyer for ordinary matroids \cite{FinkSpeyer}. The audience we are aiming at is the union of combinatorists, algebraists and algebraic geometers, not the intersection.

Matroids are nowadays central objects in combinatorics. Just as groups abstract the notion of symmetry, matroids abstract the notion of \emph{independence}. The interplay of matroids and geometry is in fact already a classical subject \cite{GGMS}. Just one of such interactions (central for our article) is the following set of associations:

\begin{center}
matroids/flag matroids/polymatroids $\rightarrow$ lattice polytopes $\rightarrow$ toric varieties.
\end{center}

We describe these constructions in detail. They allow to translate results in combinatorics to and from algebraic geometry. As an example we discuss two ideas due to White:
\begin{itemize}
\item combinatorics of basis covers translates to projective normality of (all maximal) torus orbit closures in arbitrary Grassmannians - Theorem \ref{thm:White},
\item White's conjecture about basis exchanges (Conjecture \ref{conj:White}) is a statement about quadratic generation of ideals of toric subvarieties of Grassmannians.
\end{itemize}
Although the idea to study a matroid through the associated lattice polytope is certainly present already in the works of White and Edmonds, the importance of this approach was only fully discovered by Gelfand, Goresky, MacPherson and Serganova \cite{GGMS}. The construction of associating a toric variety to a lattice polytope can be found in many sources, we refer the reader e.g. to~\cite{Fulton, sturmfels1996grobner, Cox}.

The object we focus on is one of the main invariants of a matroid: the Tutte polynomial. It is an inhomogeneous polynomial in two variables. On the geometric side it may be interpreted as a cohomology class (or a class in $K$-theory or in Chow ring) in a product of two projective spaces.

The applications of algebro-geometric methods are currently flourishing. A beautiful result of Huh confirming a conjecture by Read on unimodality  of chromatic polynomials of graphs is based on Lefschetz theorems \cite{huh2012milnor}. This lead further to a proof of the general Rota-Heron-Welsh conjecture \cite{AdiprasitoHuhKatz} -- which we state in Theorem \ref{thm:AHK}. Although the latter proof is combinatorial in nature, the authors were inspired by geometry, in particular Lefschetz properties.
We would like to stress that the varieties and Chow rings studied by Adiprasito, Huh and Katz are not the same as those we introduce in this article. Still, as the focus of both is related to the Tutte polynomial it would be very interesting to know if their results can be viewed in the setting discussed here.

We finish the article with a few open questions. As the construction of the Tutte polynomial we propose is quite involved it would be very nice to know more direct, combinatorial properties and definitions.


\subsection{Notation}
$E$ will always denote a \emph{finite} set  of cardinality $n$. $\mathcal{P}(E)$ is the set of all subsets of $E$, and ${E\choose{k}}$ is the set of all subsets of $E$ of cardinality $k$.

We use $[n]$ as a shorthand notation for the set $\{1,2,\ldots,n\}$.

We will denote the difference of two sets $X$ and $Y$ by $X-Y$. This does not imply that $Y \subseteq X$. If $Y$ is a singleton $\{e\}$, we write $X-e$ instead of $X-\{e\}$.

\section{Matroids: combinatorics}
For a comprehensive monograph on matroids we refer the reader to \cite{oxley}.
\subsection{Introduction to matroids}
There exist many cryptomorphic definitions of a matroid -- it can be defined in terms of its independent sets, or its rank function, or its dependent sets, amongst others. One of the most relevant definitions for us is that of the rank function: 

\begin{definition}
 \label{rankaxioms}
A matroid $M=(E,r)$ consists of a ground set $E$ and a \emph{rank function} $r:\mathcal{P}(E)\rightarrow \Z+$ such that, for $X,Y\in\mathcal{P}(E)$, the following conditions hold:
\begin{itemize}
\item[{\rm R1}.] $r(X)\leq |X|$,
\item[{\rm R2}.] \emph{(monotonicity)} if $Y\subseteq X$, then $r(Y)\leq r(X)$, and
\item[{\rm R3}.] \emph{(submodularity)} $r(X\cup Y)+r(X\cap Y)\leq r(X)+r(Y)$.
\end{itemize}
\end{definition}

We write $r(M)$ for $r(E)$. When $r(X)=|X|$, we say that $X$ is \emph{independent}, and \emph{dependent} otherwise. A minimal dependent set is called a \emph{circuit}. A matroid is \emph{connected} if and only if any two elements are contained in a common circuit. It can be shown 
that ``being contained in a common circuit" is an equivalence relation on $E$; the equivalence classes are called \emph{connected components}.

If $|X|=r(X)=r(M)$ we call $X$ a \emph{basis} of $M$.
We can use bases to provide an alternative set of axioms with which to define a matroid. We present this as a lemma, but it can just as well be given as the definition. The diligent reader can check that each set of axioms implies the other.
\begin{lemma}
	\label{matroid}
	A \emph{matroid} $M=(E,\mathcal{B})$ can be described by a set $E$ and a collection of subsets $\mathcal{B}\subseteq \mathcal{P}(E)$ such that:
	\begin{itemize}
		\item[{\rm B1}.] $\mathcal{B}\neq\varnothing$, and
		\item[{\rm B2}.] \emph{(basis exchange)} if $B_1,B_2\in\mathcal{B}$ and $e\in B_1-B_2$, there exists $f\in B_2-B_1$ such that $(B_1-e)\cup f\in\mathcal{B}$.
	\end{itemize}
\end{lemma}
A reader new to matroid theory should not be surprised by the borrowed terminology from linear algebra: matroids were presented as a generalisation of linear independence in vector spaces in the paper by Whitney \cite{whitney} initiating matroid theory. Matroids also have a lot in common with graphs, thus explaining even more of the terminology used. For instance, very important matroid operations are that of \emph{minors}. These are analogous to the graph operations of the same names. As there, deletion is very simple, while contraction requires a bit more work.

\begin{definition}[Deletion and Contraction]
\label{delete}\
\begin{itemize}
\item We can remove an element $e$ of a matroid $M=(E,r)$ by \emph{deleting} it. This yields a matroid $M\backslash e=(E-e,r_{M\backslash e})$, where $r_{M\backslash e}(X)=r_M(X)$ for all $X\subseteq E-e$.
\item We can also remove an element $e$ of a matroid $M=(E,r)$ by \emph{contracting} it. This gives a matroid $M/e=(E-e,r_{M/e})$ where $r_{M/e}(X)=r_M(X\cup e)-r_M(\{e\})$ for all $X\subseteq E-e$.
\end{itemize}
\end{definition}

\begin{remark}
	More generally, if $M=(E,r)$ is a matroid and $S$ is a subset of $E$, we can define the deletion $M\backslash S$ (resp.\ contraction $M/S$) by deleting (resp.\ contracting) the elements of $S$ one by one. We have that $r_{M\backslash S}(X)=r_M(X)$ for all $X\subseteq E-S$ and $r_{M/S}(X)=r_M(X\cup S)-r_M(S)$ for all $X\subseteq E-S$.
\end{remark}

We will now give two examples of classes of matroids which show exactly the relationship matroids have with linear algebra and graph theory. The first one plays a central role in our article. 



 \begin{definition}
 		Let $V$ be a vector space, and $\phi: E \to V$ a map that assigns to every element in $E$ a vector of $V$. For every subset $X$ of $E$, we define $r(X)$ to be the dimension of the linear span of $\phi(X)$, then $(E,r)$ is a matroid, which we say is \emph{representable}.
 \end{definition}
\begin{remark}
	Our definition differs slightly from the one found in literature: typically one identifies $E$ with $\phi(E)$. Our definition does not require $\phi$ to be injective; we can take the same vector several times.
	We also note that the matroid represented by $\phi: E \to V$ only depends on the underlying map $\phi: E \to \PP(V)$, assuming $\phi(E)\subset V\setminus\{0\}$. 
\end{remark}


If $V$ is defined over a field $\F$, we say that $M$ is \emph{$\F$-representable.}
We can describe the bases of a representable matroid: $X \subseteq E$ is a matroid basis  if and only if $\phi(X)$ is a vector space basis of the linear span of $\phi(E)$.
\begin{example}[The non-Pappus matroid] 
\label{eg:nonpappus}
	Here is an example of a non-representable matroid: consider the rank-3 matroid $R$ on $[9]$, whose bases are all 3-element subsets of $[9]$ except for the following:
	\begin{equation*}\label{eq:nonpappus}
		 \{1,2,3\} ,\{4,5,6\}, \{1,5,7\}, \{1,6,8\}, \{2,4,7\}, \{2,6,9\}, \{3,4,8\}, \{3,5,9\}.
	\end{equation*}
	 If $R$ were representable over a field $\F$, there would be a map $[9] \to \PP_{\F}^2: i \to p_i$ such that $p_i,p_j,p_k$ are collinear if and only if $\{i,j,k\}$ is not a basis of $R$. Now, the classical \emph{Pappus' Theorem} precisely says that this is impossible: if the non-bases are all collinear, then so are $p_7,p_8,p_9$.

\begin{figure}[h]
\begin{tikzpicture}
\node (1) at (0,-0.3) {1};
\node (2) at (1.75,-0.3) {2};
\node (3) at (4,-0.3) {3};
\node (4) at (0,2.3) {4};
\node (5) at (2,2.3) {5};
\node (6) at (4,2.3) {6};
\coordinate (a) at (0,0);
\coordinate (b) at (1.75,0);
\coordinate (c) at (4,0);
\coordinate (d) at (1,1);
\node (7) at (0.7,1) {7};
\coordinate (e) at (2,1);
\node (8) at (2,1.3) {8};
\coordinate (f) at (3,1);
\node (9) at (3.3,1) {9};
\coordinate (g) at (0,2);
\coordinate (h) at (2,2);
\coordinate (i) at (4,2);
\draw (a) -- (b) -- (c)--(f)--(h)--(i)--(f)--(b)--(d)--(g)--(h)--(d)--(a)--(e)--(i);
\draw (c)--(e)--(g);
\foreach \pt in {a,b,c,d,e,f,g,h,i} \fill[black] (\pt) circle (3pt);

\end{tikzpicture}
\caption{The non-Pappus matroid}
\end{figure}
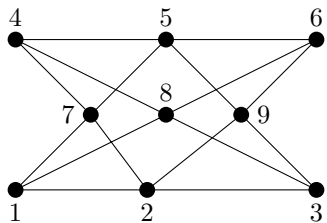
\end{example}

\begin{definition}
Let $G=(V,E)$ be a graph. The \emph{graphic} (or \emph{cycle}) matroid $M$ of $G$ is formed by taking $E(M)=E(G)$, and setting the rank of a set of edges equal to the cardinality of the largest spanning forest contained within it.
\end{definition}
Unfortunately, connectedness of the graph $G$ is not equivalent to connectedness of the matroid $M(G)$. This is obvious from looking at, for example, a tree.


\subsection{The Tutte polynomial}
Further matroid definitions will be given later, but we have covered enough to give the major object of our interest in this paper, namely the Tutte polynomial. This is the most famous matroid (and graph) invariant, and, like matroids themselves, has multiple definitions. These will be mentioned where relevant.  Here, we give the \emph{corank-nullity} formula, two terms which will be defined below.

\begin{definition}\label{dfn:Tutte}
Let $M=(E,r)$ be a matroid with ground set $E$ and rank function $r:\mathcal{P}(E)\rightarrow \Z+$. The \emph{Tutte polynomial} of $M$ is
\begin{displaymath}
    T_M(x,y) = \sum_{S\subseteq E} (x-1)^{r(M)-r(S)}(y-1)^{|S|-r(S)}.
\end{displaymath}
    \end{definition}

The term $r(M)-r(S)$ is called the \emph{corank} while the term $|S|-r(S)$ is called the \emph{nullity}. Readers familiar with matroid theory should be careful not to confuse a mention of corank with dual rank, given the usual naming convention of dual objects. By identifying the rank function of a matroid with the connectivity function of a graph in an appropriate way, one can pass between this formula and the original formulation of the Tutte polynomial which was given for graphs.
\begin{example}
	
	
	For the (matroid of the) complete graph $K_{4}$, there are four subsets with three elements of rank $2$ and all the others subsets with three elements have rank $3$. In this case, the Tutte polynomial is
	\begin{align*}
\begin{split}
T_{M(K_{4})}(x, y) & = x^{3}+3x^{2}+2x+4xy+2y+3y^{2} +y^{3}.
\end{split}
\end{align*}
\end{example}

Readers interested in seeing what the Tutte polynomial looks like for a range of different classes of matroids should consult \cite{tutteexamples}.

The prevalence of the Tutte polynomial in the literature is due to the wide range of applications it has. The simplest of these occurs when we evaluate the polynomial at certain points, these being called \emph{Tutte invariants}.  For instance, $T(1,1)$ gives the number of bases in the matroid (or the number of spanning trees in a graph). In this way we can also count the number of independent sets in a matroid or graph, and the number of acyclic orientations of a graph, as well as some other such quantities. Beyond numerics, the Tutte invariants also include other well-known polynomials, appearing in graph theory (the chromatic polynomial, concerned with graph colourings; see also Theorem \ref{thm:AHK}) and network theory (the flow and reliability polynomials). Extending to further disciplines, one can find multivariate versions of the Tutte polynomial which specialise to the Potts model \cite{potts} from statistical physics and the Jones polynomial \cite{jones} from knot theory. In this paper, we will be looking at the classical Tutte polynomial from an algebraic point of view.

We noted that there are multiple definitions of the Tutte polynomial. One is both so useful and attractive that we would be remiss to not include it. It states that, instead of calculating the full sum above, we can instead simply form a recurrence over minors of our matroid, which can lead to faster calculations. Note that a \emph{coloop} is an element of $E$ which is in every basis of $M$, while a \emph{loop} is an element which is in no basis.

\begin{lemma}[\kern-0.5em \cite{Brylawski1992}]\label{recurrance} Let $T_M(x,y)$ be the Tutte polynomial of a matroid $M=(E,r)$. Then the following statements hold.
\begin{enumerate}[label=\roman*.]
\item $T_M (x,y)=xT_{M/e}(x,y)$ if $e$ is a coloop.
\item $T_M(x,y)=yT_{M\backslash e}(x,y)$ if $e$ is a loop.
\item $T_M(x,y)=T_{M\backslash e}(x,y)+T_{M/e}(x,y)$ if $e$ is neither a loop nor a coloop.
\end{enumerate}
\end{lemma}

The Tutte polynomial is in fact universal for such formulae: any formula for matroids (or graphs) involving just deletions and contractions will be an evaluation of the Tutte polynomial.

\subsection{The base polytope} \label{sec:matroidBasePolytope}
We will now 
give two more axiom systems for matroids. The first one, via base polytopes, will play a fundamental role in this paper. 

We first define what the base polytope of a matroid is:
let the set of bases of a matroid $M=(E,r)$ be $\mathcal{B}$. We work in the vector space $\mathbb R^E=\{(r_i \ | \ i\in E)\}$, where $r_i\in\mathbb{R}$.
For a set $U\subseteq E$, ${\be}_U\in\mathbb R^E$ is the indicator vector of $U$, that is, ${\be}_U$ is the sum of the unit vectors ${\be}_i$ for all $i\in U$.  Note that ${\be}_{\{i\}} = {\be}_{i}$. 

\begin{definition}
	\label{basepolytope}
	The \emph{base polytope} of $M$ is $$P(M)=\conv\{{\be}_B \ | \ B\in\mathcal{B}\}.$$
\end{definition}

Note that this is always a lattice polytope.
The dimension of it is equal to $|E|$ minus the number of connected components of the matroid \cite[Proposition 2.4]{feichtner2005matroid}.
We also note that the vertices of $P(M)$ correspond to the bases of $M$. In particular: given $P(M)\subset \RR^{E}$, we can recover $M$.

The following theorem gives a characterisation of which lattice polytopes appear as the base polytope of a matroid. It can be used as an axiom system to define matroids:

\begin{theorem}[\kern-0.5em {\cite{edmonds}}, see also {\cite[Theorem 4.1]{GGMS}}]\label{thm:edges}
	A polytope $P \subset \mathbb{R}^E$ is the base polytope of matroid on $E$ if and only if the following two conditions hold:
	\begin{enumerate}
		\item[{\rm P1}.] every vertex of $P$ is a $0,1$-vector, and
		\item[{\rm P2}.] every edge of $P$ is parallel to $\be_i-\be_j$ for some $i,j \in E$.
	\end{enumerate}
\end{theorem}
	More generally the description of faces of matroid base polytopes is provided in \cite{kim2010flag, feichtner2005matroid}.
The base polytope is a face of the \emph{independent set polytope} of $M$, which is the convex hull of indicator vectors of the independent sets of $M$.

\subsection{Definition via Gale orderings}
We move on to another axiom system: via Gale orderings.
This definition is orginally due to Gale \cite{gale}; our formulation is based on lecture notes by Reiner \cite{Reiner}.
\begin{definition}
	Let $\omega$ be a linear ordering on $E$, which we will denote by $\leq$.
	Then the \emph{dominance ordering} $\leq_{\omega}$ on ${E\choose{k}}$, also called \emph{Gale ordering}, is defined as follows.
	Let $A,B \in {E\choose{k}}$, where
	\[
	A=\{i_1,\ldots,i_k\}, i_1 < \ldots < i_k
	\]
	and
	\[
	B=\{j_1,\ldots,j_k\}, j_1 < \ldots < j_k \text{.}
	\]
	Then we set
	\[
	A \leq_{\omega} B \text{ if and only if } i_1 \leq j_1, \ldots, i_k \leq j_k \text{.}
	\]
\end{definition}
\begin{theorem}[Gale, \cite{gale}] Let $\mathcal{B} \subseteq {E\choose{k}}$, then $\mathcal{B}$ is the set of bases of a matroid if and only if for every linear ordering $\omega$ on $E$, the collection $\mathcal{B}$ has a maximal element under the Gale ordering $\leq_{\omega}$ (i.e.\ there is a unique member $A \in \mathcal{B}$ such that $B \leq_{\omega} A$ for all $B \in \mathcal{B}$).
\end{theorem}
As well as considering certain classes of matroids, it is interesting to look at natural generalisations or extensions of matroids as a whole. In Section \ref{subsec:FlagMatroids} we will introduce \emph{flag matroids}. They will be defined via Gale orderings, generalising the above characterisation of matroids. 
In this paper, they will arise quite naturally when generalising our geometric description of matroids given in Section \ref{sec:repmat}.

\subsection{The matroid union theorem}

Next we present one of the central theorems in matroid theory.
\begin{theorem}[The matroid union theorem]\label{thm:MUT}
Let $M_1,\dots,M_k$ be matroids on the same ground set $E$ with respective families of independent sets $\mathfrak{I}_1,\dots,\mathfrak{I}_k$ and rank functions $r_1,\dots,r_k$. Let $$\mathfrak{I}:=\{I\subset E: I=\bigcup_{i=1}^k I_i\text{ for }I_i\in\mathfrak{I}_i\}.$$
Then $\mathfrak{I}$ is also a family of independent sets for a matroid, known as the union of $M_1,\dots,M_k$. Further, the rank of any set $A\subset E$ for the union matroid is given by:
$$r(A)=\min_{B\subset A}\{|A\setminus B|+\sum_{i=1}^k r_i(B)\}.$$
\end{theorem}
The proof can be found e.g.~in \cite[12.3.1]{oxley}. The following corollary is essentially due to Edmonds.
\begin{corollary}\label{cor:cover}
Let $M_1,\dots,M_k$ be matroids on a ground set $E$ with rank functions respectively $r_1,\dots,r_k$. $E$ can be partitioned into independent sets, one for each matroid, if and only if for all subsets $A\subset E$ we have $|A|\leq  \sum_{i=1}^k r_i(A)$.
\end{corollary}
\begin{proof}
The implication $\Rightarrow$ is straightforward.

For the other implication let $U$ be a matroid that is the union of $M_1,\dots,M_k$. 
We compute the rank of $E$ in $U$, applying the matroid union theorem \ref{thm:MUT}:
$$r_U(E)=\min\{|E|-|B|+\sum_{i=1}^k r_i(B)\}.$$
By assumption for any $B\subset E$ we have $|E|-|B|+\sum_{i=1}^k r_i(B)\geq |E|$. Further, equality holds for $B=\emptyset$. Hence, $r_U(E)=|E|$. Thus, $E$ is an independent set of $U$. By definition it is a union of $k$ independent sets, one in each of the $M_i$'s.
\end{proof}

\section{Polymatroids: combinatorics}
Consider what happens if we drop one of the rank axioms, namely that which states $r(X)\leq |X|$. What object do we get, and what relation does it have to matroids?  This object was originally studied by Edmonds  \cite{edmonds} (although in a different guise,
see Definition \ref{verySoon}), and dubbed a \emph{polymatroid}. The class of polymatroids includes, naturally, the class of matroids, and is greatly important in the field of combinatorial optimisation.

\begin{definition}
A \emph{polymatroid} $\pM=(E,r)$ consists of a ground set $E$ and a rank function $r:\mathcal{P}(E)\rightarrow \Z+$. The rank function $r$ satisfies conditions R2 (monotonicity) and R3 (submodularity) of Definition \ref{rankaxioms}, while condition R1 is relaxed to
%
$r(\emptyset)=0$.

A polymatroid is called a \emph{$k$-polymatroid} if all singletons have rank at most $k$. 
In particular, a matroid is a $1$-polymatroid.
\end{definition}
\begin{remark}
As we assume that our rank function take only integral values, the object we defined is sometimes referred in literature a \emph{discrete polymatroid} \cite{herzog}.
\end{remark}

One vital difference between matroids and polymatroids is that polymatroids do not have well-defined properties of deletion and contraction. 
One problem behind this is that, in a matroid, every basis (or the basis minus the element) is either in the deletion or contraction of any given element of the ground set. This is not true in polymatroids. In consequence, the Tutte polynomial is not directly applicable to polymatroids. In restricted cases, this can be somewhat solved: this is done by Oxley and Whittle \cite{whittle} for 2-polymatroids, 
where the corank-nullity polynomial is still universal for a form of
deletion-contraction recurrence. In \cite{alex3}, the authors strengthen the notion of ``deletion-contraction invariant" to more general combinatorial objects via the use of coalgebras, and compare their results to that of Oxley and Whittle. In their strengthening, the corank-nullity polynomial is indeed still universal, and furthermore, out of the polynomials found in \cite{whittle}, the corank-nullity one is optimal, under the norms used in \cite{alex3}.

Cameron and Fink \cite{myself} construct a version of the Tutte polynomial for all polymatroids which specialises to an evaluation of the  classical Tutte polynomial when applied to a matroid. This will be discussed below. In order to describe it, we first have to explain bases and base polytopes for polymatroids.


\begin{definition}
	Let $\pM=(E,r)$ be a polymatroid. An integer vector ${\bf x} \in \Z+^E$ is called an \emph{independent vector} if ${\bf x}\cdot{\be}_U\leq r(U) \ \mathrm{for \ all} \ U\subseteq E$. If in addition ${\bf x}\cdot{\be}_E= r(E)$, then $\bf x$ is called a \emph{basis}.
\end{definition}
Analoguously to the matroid case, we can give an axiom system for polymatroids in terms of their bases.
\begin{lemma}[\kern-0.5em {\cite[Theorem 2.3]{herzog}}]
	A nonempty finite set $\mathcal{B} \subset \Z+^E$ is the set of bases of a polymatroid on $E$ if and only if $\mathcal{B}$ satisfies
	\begin{enumerate}
		\item all $\bu \in \mathcal{B}$ have the same modulus (sum of entries), and 
		\item if $\bu = (u_1, \ldots, u_n)$ and $\bv = (v_1, \ldots, v_n)$ belong to $\mathcal{B}$ with $u_i > v_i$ then there is $j \in E$ with $u_j < v_j$ such that $\bu-\be_i+\be_j \in \mathcal{B}$.
	\end{enumerate}
\end{lemma}

\begin{definition} \label{verySoon}
Let $\pM=(E,r)$ be a polymatroid. Let $\mathcal{I}\subseteq \Z+^E$ be the set of independent vectors, and $\mathcal{B}\subseteq \Z+^E$ be the set of bases. We have the independent set polytope, which is also referred to as the \emph{extended polymatroid} of $r$:
$$EP(\pM)=\conv \mathcal{I}=\{{\bf x}\in\mathbb{R}_{\geq 0}^E \ | \  x\cdot{\be}_U\leq r(U) \ \mathrm{for \ all} \ U\subseteq E\}.$$
We also have the polymatroid \emph{base polytope}:
\[P(\pM) = \conv \mathcal{B} = EP(\pM)\cap\{x\in\mathbb{R}^E| \  x\cdot{\be}_E= r(E)\}
.\]
\end{definition}

 The  base polytope is in fact what was originally defined to be a polymatroid, by Edmonds.
As before, the base polytope is a face of the extended polymatroid. When the polymatroid considered is a matroid, these definitions coincide exactly with those from Section \ref{sec:matroidBasePolytope}.

Theorem \ref{thm:edges} generalises to the case of polymatroids, giving us another equivalent definition of polymatroids in terms of their base polytopes:
\begin{theorem}[{See \cite[Theorem 3.4]{herzog}}]
	A polytope $P \subset \mathbb{R}^n$ is the base polytope of a polymatroid on $[n]$ if and only if the following two conditions hold:
	\begin{enumerate}
		\item Every vertex of $P$ has coordinates in $\Z+$.
		\item Every edge of $P$ is parallel to $\be_i-\be_j$ for some $i,j \in [n]$.
	\end{enumerate}
\end{theorem}
If $\pM$ is a polymatroid, then the bases (resp.\ independent vectors) of $\pM$ are precisely the lattice points of $P(\pM)$ (resp.\ $EP(\pM)$). The following proposition describes which bases of $\pM$ correspond to vertices of $P(\pM)$. 
\begin{prop}[\kern-0.5em {\cite{edmonds}, see also \cite[Proposition 1.3]{herzog}}]
	Let $\pM=([n],r)$ be a polymatroid and assign an ordering $S$ to the ground set $[n]$. Let $S_i$ be the first $i$ elements according to this ordering. Every possible $S$ corresponds (not necessarily uniquely) to a vertex of P(M), ${\bf x}={\bf x}_{S}$, where ${\bf x} = (x_1,\ldots x_n)$, and
	\begin{align*}
	x_i = r(S_i)-r(S_{i-1}).
	\end{align*}
In particular, a polymatroid base polytope has at most $n!$ vertices.
\end{prop}



We finish by slightly generalising the ideas of White \cite{WHITE1977292}. As we will see later in Theorem \ref{thm:White} the statement below has geometric consequences. It was proven in \cite[Corollary 46.2c]{schrijver} using different methods.
\begin{theorem}\label{thm:sumofpoly}
	Let $\pM_1,\dots,\pM_k$ be polymatroids on a ground set $E$ with respective polytopes $P(\pM_1),\dots,P(\pM_k)$. Then every lattice point $p\in P(\pM_1)+\dots+P(\pM_k)$ is a sum $p=s_1+\dots +s_k$, where each $s_i$ is a lattice point of $P(\pM_i)$.
\end{theorem}
\begin{proof}
	Proceeding by induction on $k$, it is enough to prove the theorem for $k=2$.
	
	Let us choose $r$ large enough, so that $\pM_1$ and $\pM_2$ are $r$-polymatroids.
	We define a matroid $\tilde M_1$ (resp.~$\tilde M_2$) on $E\times [r]$ as follows. Let $\pi_1:E\times [r]\rightarrow E$ and $\pi_2:E\times [r]\rightarrow [r]$ be the projections. A subset $A\subset E\times [r]$ is independent in $\tilde M_1$ (resp.~$\tilde M_2$) if and only if for every subset $B\subset \pi_1(A)$ we have $r_1(B)\geq|\pi_2(A\cap(B\times[r]))|$ (resp.~$r_2(B)\geq|\pi_2(A\cap(B\times[r]))|$), where $r_1$ (resp.~$r_2$) is the rank function of the polymatroid $\pM_1$ (resp.~$\pM_2$). Intuitively, an independent set in $\tilde M_j$ is an independent set $I$ in $\pM_j$ where we replace one point in $E$ by as many points as the rank function dictates. We have natural surjections, $P(\tilde M_j)\rightarrow P(\pM_j)$ and $P(\tilde M_1)+P(\tilde M_2)\rightarrow P(\pM_1)+P(\pM_2)$, coming from the projection $\pi_1$. Thus, it is enough to prove the statement for two matroids. From now on we assume that $\pM_1$ and $\pM_2$ are matroids.
	
	Let $p\in P(M_1)+P(M_2)$ be a lattice point. We know that $p= \sum_i \lambda_i t_i+\sum_j \mu_j q_j$ with $\sum \lambda_i=1$, $\sum \mu_j=1$, $0\leq\lambda_i,\mu_j\in \QQ$ and $t_i$ (resp.~$q_j$) are lattice points of $P(M_1)$ (resp.~$P(M_2)$). After clearing the denominators we have
	$$dp=\sum_i \lambda_i' t_i+\sum_j \mu_j' q_j,$$
	where $\sum \lambda_i'=d$, $\sum \mu_j'=d$ and $0\leq\lambda_i',\mu_j'\in \Z+$.
	
	By restricting the set $E$ we may assume that all coordinates of $p=(p_1,\dots,p_n)$ are nonzero (i.e.\ $p_i \in \{1,2\}$), where we identify $E$ with $[n]$. 
	
	We start by defining two matroids $N_1,N_2$.
	Both are on the ground set $E_N:=\{(i,j):i\in E, 1\leq j\leq p_i\}$. In other words, we replace a point $i$ in the set $E$ by $p_i$ equivalent points. A subset $\{(i_1,j_1),\dots,(i_s,j_s)\}\subset E_N$ is independent in $N_1$ (resp.~$N_2$) if and only if
	\begin{itemize}
		\item all $i_q$'s are distinct, and
		\item $\{i_1,\dots,i_s\}$ is an independent set in $M_1$ (resp.~$M_2$).
	\end{itemize}
	
	We note that a basis of $N_1$ (resp.~$N_2$) maps naturally to a basis of $M_1$ (resp.~$M_2$). The rank function $r_{N_1}$ (resp.~$r_{N_2}$) for $N_1$ (resp.~$N_2$) is the same as the one for $M_1$ (resp.~$M_2$) if we forget the second coordinates.
	Further, the point $p$ has a decomposition as a sum of two points corresponding to basis of $M_1$ and $M_2$ if and only if the ground set $E_N$ is covered by a basis of $N_1$ and a basis of $N_2$. Hence, by Corollary \ref{cor:cover} it is sufficient to prove the following:
	
	For any $A\subset E_N$ we have $|A|\leq r_{N_1}(A)+r_{N_2}(A)$.
	
	We define a matroid $N'_1$ (resp.~$N'_2$) on the ground set $E_{N'}:=\{(i,j,l):i\in E, 1\leq j\leq p_i, 1\leq l\leq d.\}$. In other words we replace any point of $E_N$ by $d$ equivalent points.  A subset $\{(i_1,j_1,l_1),\dots,(i_s,j_s,l_s)\}\subset E_N$ is independent in $N'_1$ (resp.~$N'_2$) if only if
	\begin{itemize}
		\item all $i_q$'s are distinct, and
		\item $\{i_1,\dots,i_s\}$ is an independent set in $M_1$ (resp.~$M_2$).
	\end{itemize}
	We have a natural projection $\pi:E_{N'}\rightarrow E_N$ given by forgetting the last coordinate. We note that $r_{N'_j}(\pi^{-1}(A))=r_{N_j}(A)$ for $j=1,2$. As the point $dp$ is decomposable we know that the set $E_{N'}$ can be covered by $d$ bases of $N'_1$ and $d$ bases of $N'_2$. Hence, for any $B\subset E_{N'}$ we have:
	$|B|\leq dr_{N'_1}(B)+dr_{N'_1}(B)$. Applying this to $\pi^{-1}(A)$ we obtain:
	$$d|A|=|\pi^{-1}(A)|\leq d\cdot r_{N'_1}(\pi^{-1}(A))+d\cdot r_{N'_2}(\pi^{-1}(A))=d\cdot \left(r_{N_1}(A)+r_{N_2}(A)\right).$$
	After dividing by $d$ we obtain the statement we wanted to prove.
\end{proof}

\subsection{The Tutte polynomial for polymatroids} \label{TuttePoly}
The Tutte polynomial for polymatroids is not nearly as well-studied as in the matroid case. We gave two examples of where it was considered in certain classes of polymatroids. We will now go into detail about one suggestion how to construct Tutte polynomial in full generality. 

As mentioned, Cameron and Fink \cite{myself} form a polynomial having Tutte-like properties for polymatroids, which specialises to an evaluation of the Tutte polynomial when applied only to matroids. This is a construction which takes a polytopal, lattice-point-counting, approach as opposed to a straight combinatorial one. It is motivated by an alternative definition of the Tutte polynomial to those we have discussed so far.

\begin{definition}
Take a matroid $M=(E,r)$, and give $E$ some ordering. Let $B$ be a basis of $M$.
\begin{itemize}
\item[{i}.] We say that $e\in E-B$ is \emph{externally active} with respect to $B$ if $e$ is the smallest element in the unique circuit contained in $B\cup e$, with respect to the ordering on $E$.
\item[{ii}.] We say that $e\in B$ is \emph{internally active} with respect to $B$ if $e$ is the smallest element in the unique cocircuit in $(E- B)\cup e$. 
\end{itemize}
\end{definition}

A \emph{cocircuit} is a minimal set among sets intersecting every basis. We will not be using this notion again in the article.

We will denote the number of internally active elements with respect to $B$ with  $I(B)$ and the number of externally active elements by $E(B)$. Then we have the following result.

\begin{theorem}[\kern-0.5em \cite{tutteactivity}]
$$T_M(x,y)=\sum_{B\in\mathcal{B}} x^{I(B)}y^{E(B)}.$$
\end{theorem}

Activity was generalised to hypergraphs by K\'{a}lm\'{a}n in \cite{kalman}, where he proved that a formula similar to the one above does not hold for hypergraphs. The one-variable specialisations are, however, consistent. That is, $T(x,0),T(0,y)$ can be written in terms of activity generating functions for hypergraphic polymatroids.  In \cite{myself}, the authors show that this behaviour extends to all polymatroids given their own Tutte-like polynomial for polymatroids. Their construction is as follows.


Let $\Delta$ be the standard simplex in $\mathbb R^E$ of dimension equal to $|E|-1$,
and $\nabla$ be its reflection through the origin. Construct the polytope given by the Minkowski sum $P(\pM)+u\Delta+t\nabla $ where $\pM=(E,r)$ is any polymatroid and $u,t\in\Z+$. By Theorem 7 of \cite{mcmullen}, the number of lattice points inside the polytope is a polynomial in $t$ and~$u$, of degree $\dim(P(\pM)+u\Delta+t\nabla) = |E|-1$.
This polynomial is written in the form
\begin{equation*}
\label{qdef}
Q_\pM(t,u):=\#(P(\pM)+u\Delta+t\nabla )=\sum_{i,j} c_{ij}\binom{u}{j}\binom{t}{i}.\end{equation*}
Changing the basis of the vector space of rational polynomials gives the polynomial \begin{equation*}
\label{q'def}
Q'_\pM(x,y)=\sum_{ij}c_{ij}(x-1)^i(y-1)^j
\end{equation*}

As mentioned, this specialises to the Tutte polynomial:

\begin{citethm}[{\kern-0.5em \cite[Theorem 3.2]{myself}}]
\label{TtoQ}
Let $M=(E,r)$ be a matroid. Then we have
\begin{equation*}
Q'_M(x,y)=\dfrac{x^{|E|-r(M)}y^{r(M)}}{x+y-1}\cdot T_M\left(\dfrac{x+y-1}{y},\dfrac{x+y-1}{x}\right)
\end{equation*}
as an identification of rational functions.
\end{citethm}

Formulae can be given for $Q'_\pM$
under the polymatroid generalisation of many standard matroid operations such as direct sum and duality, and in many cases the results are versions of those true for the Tutte polynomial. Most importantly, this polynomial also satisfies a form of deletion-contraction recurrence.

\begin{citethm}[{\kern-0.5em \cite[Theorem 5.6]{myself}}]
\label{delcont}
Let $\pM=(E,r)$ be a polymatroid and take $a\in E(\pM)$. Let $\pN_k$ be the convex hull of $\{p\in P(\pM) \ | \ p_a=k\}$ for some $k\in\{0,\ldots,r(\pM)\}$, and $Q'_\pN$ the polynomial formed by replacing the $\pM$ in the above definition by $\pN$. Then $$Q'_\pM(x,y)=(x-1)Q'_{\pM\backslash a}(x,y)+(y-1)Q'_{\pM/a}(x,y)+\sum_{k} Q'_{\pN_k}(x,y).$$
\end{citethm}

\begin{remark} \label{rmk:polymatroidmatroid}
	We mention another approach of generalising the Tutte polynomial to polymatroids. In the proof of Theorem \ref{thm:sumofpoly}, we explained how to associate to an $r$-polymatroid $\pM$ on $E$ a matroid $M$ on $E \times [r]$. We could define the Tutte polynomial of $\pM$ to be simply the usual Tutte polynomial of $M$. Of course, the result might depend on the chosen $r$. Still, it is natural to ask if there is any relation between this construction and the one described above.
\end{remark}
\section{Flag varieties: geometry} \label{sec:flagVar}
In this section, we always work over the field of complex numbers. 
\subsection{Representations and characters}\label{sub:repsandchar}
We begin by fixing some notation regarding the representation theory of $GL_n$. We refer the reader to \cite[Chapter 4]{SturmfelsInvariant} for a brief introduction to the representation theory of $GL_n$, or to \cite{FultonHarris} for a more detailed account.

The (polynomial) irreducible representations of $GL_n$ are in bijection with Young diagrams $\lambda$ with at most $n$ rows. We write $\lambda = (a_1,\ldots,a_k)$ for the Young diagram with rows of length $a_1\geq a_2 \geq \ldots \geq a_k > 0$. The associated $GL_n$-representation is called a \emph{Weyl module of highest weight $\lambda$}, and will be denoted by $\SS^\lambda V$ (here $V$ refers to the natural representation of $GL_n$, i.e.~an $n$ dimensional vector space with the linear $GL_n=GL(V)$ action). The usual construction of Weyl modules goes via Young symmetrisers, which we will not recall here. For readers not familiar with them, it suffices to know that $\SS^{(a)}V=S^aV$, the $a$-th symmetric power of the natural representation, and that $\SS^{(1,\ldots,1)}V$ (where $1$ appears $a$ times), is the exterior power $\bigwedge^aV$.

In this article we will be interested in the action of a maximal torus $T\subset GL(V)$ on flag varieties. Such a torus $T\simeq (\K^*)^n$ may be identified with the diagonal nondegenerate matrices, after fixing a basis of $V$. We recall that a torus $T$ acting on any vector space $W$ induces a weight decomposition:
$$W=\bigoplus_{\bf c\in\ZZ^n} W_{\bf c},$$
where $(t_1\dots,t_n)\in T$ acts on $v\in W_{\bf c}$ by scaling as follows:
$$(t_1,\dots,t_n)v=t_1^{c_1}\cdots t_n^{c_n} v.$$
In particular, an action of the torus $T$ on a one-dimensional vector space $\K$ may be identified with a lattice point in $\ZZ^n$.
We call $\ZZ^n=M$ the lattice of \emph{characters} of $T$. An element $(a_1,\dots,a_n)\in M$ is a character identified with the map:
$$T\ni(t_1,\dots,t_n)\mapsto t_1^{a_1}\cdots t_n^{a_n}\in \K^*.$$

Any irreducible $GL(V)$-representation $W=\SS^{c_1,\dots,c_n} V$ decomposes as above under the action of $T$ with a single one-dimensional component $W_{c_1,\dots,c_n}$; moreover all other components have a lexicographically smaller weight.
 This explains the name ``Weyl module of highest weight $({c_1,\dots,c_n})$."

\subsection{Grassmannians}
A basic example of a flag variety is a Grassmannian $G(k,V)$, which as a set parameterises $k$-dimensional subspaces of an $n$-dimensional space $V$.
A point in $G(k,n)$ can be represented by a full-rank $k \times n$-matrix $A$, where our $k$-dimensional subspace is the row span of $A$. Two matrices $A$ and $B$ represent the same point in $G(k,n)$ if and only if they are the same up to elementary row operations.

 $G(k,n)$ can be realised as an algebraic variety as follows:
$$G(k,n)=G(k,V)=\{[v_1\wedge\dots\wedge v_k]\subset \PP(\bigwedge^k V)\}.$$
Here, $v_1,\dots, v_k$ are the rows of the aforementioned matrix $A$, and thus
a point of $G(k,n)$ is identified with the subspace $\langle v_1,\dots,v_k\rangle$.
The embedding presented above is known as the \emph{Pl\"ucker embedding} and the Grassmannian is defined by quadratic polynomials known as \emph{Pl\"ucker relations} \cite{manivel}. Explicitly, in coordinates, the map associates to the matrix $A$ the value of all $k\times k$ minors. We refer the readers not familiar with algebraic geometry, and in particular Grassmannians, to a short introduction in \cite{MMGr}.

  The Pl\"ucker embedding may be identified with a very ample line bundle on $G(k,n)$, which we will denote by $\mathcal{O}(1)$. Other very ample line bundles on $G(k,n)$ are the $d$-th tensor powers $\mathcal{O}(d)$. They can be realised as a composition of the Pl\"ucker embedding with the $d$-th Veronese map 
$
\PP(\bigwedge^k V) \to \PP(\Sym^d \bigwedge^k V).$

\begin{remark}
A reader not familiar at all with very ample line bundles should not think about them in terms of geometric objects, but rather as maps into projective spaces. Let us present this with the example of the projective space $\PP^n$ (which also equals $G(1,n+1)$). Here we have a basic identity map $\PP^n\rightarrow \PP^n$, which we denote by $\mathcal{O}(1)$. The $r$-th \emph{Veronese map} embeds $\PP^n$ in a larger projective space $\PP^{{{n+r}\choose n}-1}$ by evaluating on a point all degree $r$ monomials. The associated map is given by $\mathcal{O}(r)$. For $n=1$ and $r=2$ we get:
\[\PP^1\ni[x:y]\rightarrow [x^2:xy:y^2]\in\PP^2.\]
\end{remark}

 It will follow from Proposition \ref{prop:BW} that the embedding of $G(k,n)$ by $\mathcal{O}(d)$ spans a projectivisation of an irreducible representation $V_{\lambda_0}$ of $GL_n$. 
 The Young diagram $\lambda_0 = (d, \ldots, d)$ 
consists of $k$ rows of length $d$. 

\subsection{Flag varieties}
More generally, for any irreducible representation $V_\lambda$ of $GL_n$ the projective space $\PP(V_\lambda)$ contains a unique closed orbit, known as a homogeneous variety or more precisely a flag variety. To describe it, let us fix a sequence of $s$ positive integers $0<k_1<\dots<k_s<n$. The flag variety is defined as follows:
$$Fl(k_1,\dots,k_s;n)=\{V_1\subset\dots\subset V_s\subset V:\dim V_i=k_i\}\subset G(k_1,V)\times\dots \times G(k_s,V).$$
Hence, flag varieties are in bijection with (nonempty) subsets of $\{1,\dots,n-1\}$, while Grassmannians correspond to singletons.\\
From now on we will abbreviate the tuple $(k_1,\ldots,k_s)$ to ${\bf k}$, and the flag variety $Fl(k_1,\ldots,k_s;n)$ to $Fl({\bf k},n)$.
A point in $Fl(\bk,n)$ can be represented by a full-rank $n \times n$-matrix $A$: the row span of the first $d_i$ rows is $V_i$. (Although note that only the first $k_s$ rows of the matrix are relevant.) As with Grassmannians, different matrices can represent the same point in $Fl(\bk,n)$. More precisely, if we partition the rows of $A$ into blocks of size $k_1,k_2-k_1,\ldots,n-k_s$, then we are allowed to do row operations on $A$, with the restriction that to a certain row we can only add a multiple of a row in the same block or a block above. Another way to think about this is the following: let $P_{\bk} \subset GL_n(\K)$ be the parabolic subgroup of all invertible matrices $A$ with $A_{ij}=0$ if $i \leq k_r < j,$ for some $r$. Then two $n \times n$ matrices represent the same flag if and only if they are the same up to left multiplication with an element of $P_{\bk}$. Hence $Fl(\bk,n)$ can also be described as the quotient $_{P_{\bk}} \backslash ^{GL_n(\K)}$ (a homogeneous variety). 

Just as for Grassmannians we may study different embeddings of flag varieties. The natural one is given by the containment
$$G(k_1,V)\times\dots \times G(k_s,V)\subset \PP(\bigwedge^{k_1}V)\times\dots\times\PP(\bigwedge^{k_s} V)\subset \PP(\bigwedge^{k_1}V\otimes\dots\otimes \bigwedge^{k_s} V),$$
where the last map is the Segre embedding. The representation $\bigwedge^{k_1}V\otimes\dots\otimes \bigwedge^{k_s} V$ in general is reducible -- precise decomposition is known by Pieri's rule (or more generally by the Littlewood-Richardson rule), see for example \cite[Proposition 15.25]{FultonHarris}. As we will prove below, the flag variety spans an irreducible representation with the corresponding Young diagram with $s$ columns of lengths $k_s,k_{s-1},\dots,k_1$ respectively.

Other embeddings can be obtained as follows. We replace the Segre map by the Segre--Veronese, i.e.~we first re-embed a Grassmannian with a Veronese map. Thus, a flag variety with an embedding can be specified by a function:
$$f:\{1, \ldots, n-1\}\rightarrow\Z+.$$
To obtain a flag variety abstractly we first consider a subset $\{a \in [n-1]: f(a)>0\}$. The Segre--Veronese map is specified by the values of the function $f$ -- the case $f(a)=1$ corresponds to the usual Segre. The irreducible representation we obtain has an associated Young diagram with $f(a)$ columns of length $a$. Before proving all these statements we present an example.
\begin{example}\label{ex:flag}
Let us fix $n=4$ and a function $f$ that on $1,2,3$ takes values $2,1,0$ respectively. The corresponding flag variety equals $Fl(1,2;4)$, i.e.~it parameterises one-dimensional subspaces, $l$, inside two-dimensional subspaces, $S$, in a fixed four-dimensional space, $V$. The line $l$ corresponds to a point in a projective space $\PP(V)=G(1,V)$ and the space $S$ to a point in $G(2,V)$. If we suppose $l=[v_1]$, then we may always find $v_2\in V$ such that $S=\langle v_1,v_2\rangle$. Hence
$$Fl(1,2)=\{[v_1]\times[v_1\wedge v_2]\in \PP(V)\times G(2,V)\}.$$
We now pass to the embedding. As $f(1)=2$ we have to consider the second Veronese map $\PP(V)\rightarrow \PP(S^2(V))$ given by $[v]\rightarrow [v\cdot v]$. We obtain
$$Fl(1,2)=\{[v_1\cdot v_1]\times [v_1\wedge v_2]\in \PP(S^2V)\times G(2,V)\}\subset \PP(S^2V\otimes \bigwedge^2 V).$$

By Pieri's rule, we have a decomposition of $GL(V)$ 
representations:
$$S^2V\otimes \bigwedge^2 V=\SS^{3,1}V\oplus \SS^{2,1,1}V.$$
 Hence, $\SS^{3,1}(V)$ corresponds to the Young diagram with the first row of length three and the second row of length one. We note that this diagram indeed has $2$ columns of length $1$, $1$ column of length $2$, and $0$ columns of length $3$. 

The flag variety is always contained in the lexicographically-first (highest weight) irreducible component -- cf.~Proposition \ref{prop:BW} below; in our example this is $\SS^{3,1}(V)$. In particular, we may realise the representation $\SS^{3,1}(V)$ as a linear span of the affine cone over the flag variety:
$$\big\langle \widehat{Fl(1,2)}=\{(v_1\cdot v_1)\otimes (v_1\wedge v_2):v_1,v_2\in V\}  \big\rangle\subset S^2V\otimes\bigwedge^2 V.$$
For readers not familiar with the construction of $\SS^{\lambda}V$, this can be taken as a definition. For a proof that this definition is equivalent to the usual construction, see Proposition \ref{prop:BW} below.
\end{example}
The function $f$ is often represented on a Dynkin diagram:

\begin{center}
\includegraphics[scale=0.2]{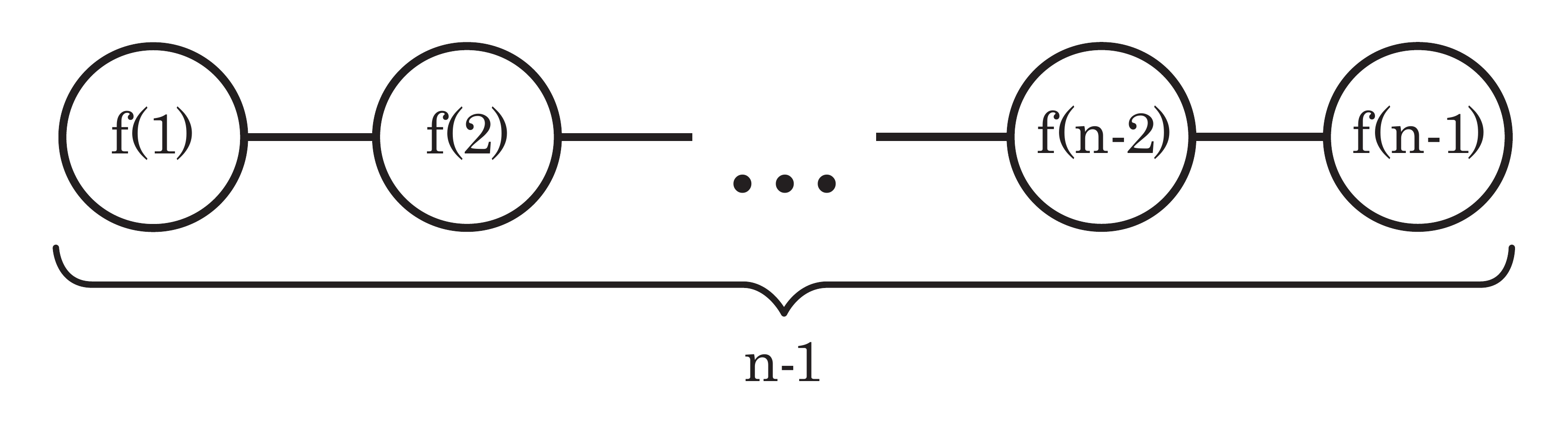}
\end{center}

For Example \ref{ex:flag} this would be:

\begin{center}
\includegraphics[scale=0.15]{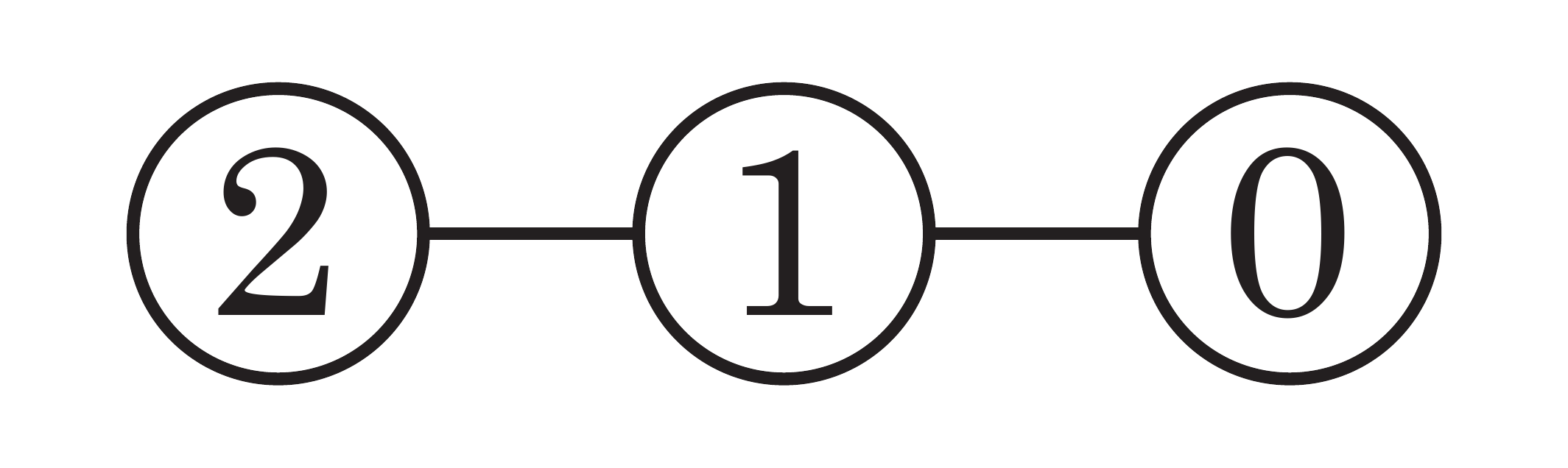}
\end{center}

We note that if $f$ is positive we obtain a complete flag variety, i.e.~the variety parametrizing complete flags:
$$V_1\subset V_2\subset\dots\subset V_{n-1}\subset \K^n.$$
The complete flag variety maps to any other flag variety, simply by forgetting the appropriate vector spaces. We note that all our constructions are explicit and only use exterior (for Grassmannians), symmetric (for Veronese) and usual (for Segre) tensor products, as in Example \ref{ex:flag}.

 We are now ready to prove a special case of the Borel-Bott-Weil Theorem relating representations and embeddings of flag varieties. 
\begin{prop}[Borel-Weil]\label{prop:BW}
Any flag variety $Fl(k_1,\dots,k_s;n)$ with an embedding given by a function $f$ 
spans the irreducible $GL(V)$-representation $S^{\lambda}V$, where the Young diagram $\lambda$ has $f(j)$ columns of length $j$.
\end{prop}
\begin{proof}
Fix a basis $e_1,\dots,e_n$ of $V$.
Let us consider the following flag of subspaces:
$$\langle e_1,\dots,e_{k_1}\rangle\subset \langle e_1,\dots,e_{k_2}\rangle\subset\dots \subset \langle e_1,\dots,e_{k_s}\rangle$$
and the corresponding point $p\in Fl(k_1,\dots,k_s)$. Under the embedding specified by $f$ it is mapped to:
$$(e_1\wedge\dots\wedge e_{k_1})^{\circ f(k_1)}\otimes\dots\otimes (e_1\wedge\dots\wedge e_{k_s})^{\circ f(k_s)}\subset S^{f(k_1)}(\bigwedge^{k_1}V)\otimes\dots\otimes S^{f(k_s)}(\bigwedge^{k_s}V).$$
The $GL(V)$-decomposition of the ambient space is highly non-trivial. However, looking directly at the $T$ decomposition we see that, up to scaling, the image of $p$ is the unique lexicographically-highest vector. Hence, in particular, the image of $p$ belongs to $S^{\lambda}V$, as all other $GL(V)$-representations appearing in the decomposition have strictly smaller highest weights. Furthermore, the flag variety is an orbit under the $GL(V)$-action -- one can explicitly write a matrix mapping any flag to any other given flag. Thus, if one point is contained in the irreducible representation, the whole variety must be contained in it.

It remains to show that the span of the flag variety is indeed the whole irreducible representation. This is true, as the flag variety is $GL(V)$-invariant, and thus its linear span is a representation of $GL(V)$. As  $S^{\lambda}V$ is irreducible, the linear span must coincide with it.
\end{proof}
The above theorem may be regarded as a realisation of irreducible $GL(V)$-representations as spaces of sections of a very ample line bundle on a flag variety. A more general Borel-Bott-Weil theorem provides not only a description of global sections -- zeroth cohomology -- but also higher, arbitrary cohomology.

Later on, we will denote a flag variety together with the embedding given by $f$ as $Fl({\bf k};n)$, where ${\bf k}=(k_1,\ldots,k_s)$ satisfies $0 < k_1\leq \ldots \leq k_s < n$ and has $f(a)$ entries equal to $a$, for all $a \in [n-1]$. For example, the embedding of Example \ref{ex:flag} will be written as $Fl(1,1,2;4)$.

\section{Representable matroids: combinatorics and geometry}\label{sec:repmat}
Let us consider a representable matroid $M$ given by $n=|E|$ vectors spanning a $k$-dimensional vector space $V$. By fixing a basis of $V$ we may represent this matroid as a $k \times n$ matrix $A$. On the other hand the matrix $A$ may be regarded as defining a $k$-dimensional subspace of an $n$-dimensional vector space, i.e.~a point in $G(k,n)$. Since applying elementary row operations to $A$ does not change which of the maximal minors of $A$ vanish, the matroid $M$ only depends on the $k$-dimensional subspace, and not on the specific matrix $A$ representing our subspace. In this way we have associated to any point $p \in G(k,n)$ a representable rank-$k$ matroid $M_p$ on $[n]$.
\begin{remark}
In the literature, the correspondence between points in $G(k,n)$ and vector arrangements in $\K^k$ is known as the \emph{Gelfand-MacPherson correspondence}. The way we just constructed it is very explicit, but has the disadvantage of not being canonical (it depends on a chosen basis of $\K^n$). There are several ways to fix this.

One way of obtaining a more intrinsic construction is to replace the Grassmannian $G(k,n)$ by the Grassmannian $G(n-k,n)$. If $A$ represents a surjective linear map from $\K^n$ to $V$, then to $A$ one can associate the $(n-k)$-dimensional kernel of this map, i.e.~a point in $G(n-k,n)$. As this construction requires dual matroids we decided to present the one above in coordinates.

A different (but closesly related) intrinsic construction would be to define $G(k,\CC^n)$ as the space of $k$-dimensional quotients of $\CC^n$ (instead of $k$-dimensional subspaces). Then $A$ maps the standard basis of $\CC^n$ to $n$ vectors in a smaller space $V \in G(k,\CC^n)$; these $n$ vectors represent a matroid.

A third solution would be to talk about hyperplane arrangements instead of vector arrangements.
\end{remark}

The vector space $\K^n$ comes with the action of a torus $T=(\K^*)^n$.
We have associated a point $p\in G(k,\K^n)$ to a representation of a matroid. If we change the representation by rescaling the vectors we do not change the matroid and the associated point belongs to the orbit $Tp$. Hence, the intrinsic properties of the matroid $M_p$ should be related to the geometry of $Tp$ -- a feature we will examine in detail throughout the article. The closure $\overline{Tp}$ is a projective toric variety. For more information about toric geometry we refer to \cite{Cox, sturmfels1996grobner, Fulton, michalek2018selected}.
\begin{remark}
Of course it can happen that different torus orbits give rise to the same matroid: there are only finitely many matroids on $[n]$, but if $1<k<n-1$ there are infinitely many torus orbits in $G(k,n)$. In fact, the set of all points in $G(k,n)$ giving rise to the same matroid forms a so-called \emph{thin Schubert cell} or \emph{matroid stratum}, which typically is a union of infinitely many torus orbits. Thin Schubert cells were first introduced in \cite{GGMS}. Thin Schubert cells are badly behaved in general: for fixed $k \geq 3$ the thin Schubert cells of $G(k,n)$ exhibit arbitrary singularities if $n$ is large enough. This is a consequence of Mn\"ev's theorem \cite{Mnev}. See \cite[Section 1.8]{Lafforgue} for a more detailed discussion.
\end{remark}
\begin{theorem}[Gelfand-Goresky-MacPherson-Serganova \cite{GGMS}]\label{thm:PM}
The lattice polytope representing the projective toric variety $\overline{Tp}$  described above is isomorphic to $P(M_p)$. 
\end{theorem}
\begin{proof}
Let $A$ be the matrix, whose rows span the space corresponding to $p$.
The parameterisation of $\overline{Tp}$ is given by:
\[\phi:T\rightarrow \PP(\bigwedge^k \K^n).\]
The coordinates of the ambient space are indexed by $k$-element subsets of the $n$ columns of the matrix $A$. The coordinate indexed by $I$ of $\phi(t_1,\dots,t_n)$ equals $\prod_{i\in I} t_i$ times the $k\times k$ minor of $A$ determined by $I$. This is nonzero if and only if $I$ is a basis of $M$. Hence, the ambient space of $\overline{Tp}$ has coordinates indexed by basis elements of $M$. Further, up to acting by a diagonal invertible isomorphism inverting the nonzero minors, the map is given by monomials, corresponding to products of $t_i$'s that are in a basis. This is exactly the construction of the toric variety represented by $P(M)$.
\end{proof}
It is a major problem to provide the algebraic equations of $\overline{Tp}$. This is equivalent to finding integral relations among the basis of a matroid. We point out that matroids satisfy a `stronger' property then one could expect from the basis exchange axiom B2 of Lemma \ref{matroid}. Precisely, for any two bases $B_1,B_2\in\mathcal{B}$ and a subset $A\subset B_1-B_2$, there exists $A'\subset B_2-B_1$ such that $(B_1-A)\cup A'$ and $(B_2-A')\cup A\in\mathcal{B}$ \cite{multexchange}. 
This exactly translates to a binomial quadric (degree 2 polynomial) in the ideal of $\overline{Tp}$:
$x_{B_1}x_{B_2}-x_{(B_1-A)\cup A'}x_{(B_2-A')\cup A}$, where, as in the proof of Theorem \ref{thm:PM}, we label the coordinates by a basis of the matroid. Further, if $|A|=1$ we obtain special quadrics corresponding to exchanging one element in a pair of basis.
The following conjecture due to White provides a full set of generators for any matroid $M$.
\begin{conjecture} \label{conj:White}
The ideal of the toric variety represented by $P(M)$ is generated by the special quadrics corresponding to exchanging one element in a pair of basis.
\end{conjecture}
We note that neither is it known that the ideal of this toric variety is generated by quadrics nor that all quadrics are spanned by the special quadrics described above. However, it is known that the special quadrics define the variety as a set (or more precisely as a projective scheme) \cite{jaMichal, lason2016toric}.

The combinatorial methods can be used to prove geometric properties of torus orbit closures in Grassmannians. Below we recall a definition of a normal lattice polytope.
\begin{definition}\label{Normal Polytope}
A lattice polytope $P$, containing $0$ and spanning (as a lattice) the lattice $N$ is \emph{normal} if and only if for any $k\in\Z+$ and any $p\in kP\cap N$ we have $p=p_1+\dots+p_k$ for some $p_i\in P\cap N$.
\end{definition}
Normality of a polytope is a very important notion as it corresponds to \emph{projective normality} of the associated toric variety \cite[Chapter 2]{Cox}, \cite{sturmfels1996grobner} (less formally, the associated toric variety is not very singular and is embedded in a particularly nice way in the projective space).
\begin{theorem}[White]\label{thm:White}
For any matroid $M$ the polytope $P(M)$ is normal. In particular, any torus orbit closure in any Grassmannian is projectively normal.
\end{theorem}
\begin{proof}
This is a special case of Theorem \ref{thm:sumofpoly}, where we take all $M_i$ equal to $M$.
\end{proof}



\section{Introduction to flag matroids}


In Section \ref{sec:repmat} we explained a correspondence between torus orbits in a Grassmannian (geometric objects) and representable matroids (combinatorial objects). We will generalise this correspondence in different ways. For instance, on the geometry side, we can replace Grassmannians with flag varieties. On the combinatorics side, this naturally leads to the notion of a (representable) \emph{flag matroid}.
Flag matroids first arose as a special case of the so-called \emph{Coxeter matroids}, introduced by Gelfand and Serganova \cite{gelfand1, gelfand2}. In this section we first give a combinatorial introduction to flag matroids. Afterwards, we explain how they are related to flag varieties. The exposition is largely based on Chapter 1 of \cite{borovik}.
\subsection{Flag matroids: Definition} \label{subsec:FlagMatroids}
We start by defining flag matroids in the way they are usually defined in the literature: using Gale orderings.
\begin{definition}
	Let $0 < k_1 \leq \ldots \leq k_s < n$ be natural numbers. Let $\bk=(k_1,\ldots,k_s)$. A \emph{flag $F$ of rank $\bk$ on $E$} is an increasing sequence
	\[
	F^1 \subseteq F^2 \subset \cdots \subseteq F^s
	\]
	of subsets of $E$ such that $|F^i|=k_i$ for all $i$.
	The set of all such flags will be denoted by $\mathcal{F}^{\bk}_E$.
\end{definition}
Let $\omega$ be a linear ordering on $E$. We can extend the Gale ordering $\leq_{\omega}$ to flags:
\[
(F^1,\ldots,F^s) \leq_{\omega} (G^1,\ldots,G^s) \text{ if and only if } F^i \leq_{\omega} G^i \text{ for all } i\text{.}
\]
\begin{definition}
	A \emph{flag matroid of rank $\bk$ on $E$} is a collection $\mathcal{F}$ of flags in $\mathcal{F}^{\bk}_E$, which we call \emph{bases} satifying the following property: for every linear ordering $\omega$ on $E$, the collection $\mathcal{F}$ contains a unique element which is maximal in $\mathcal{F}$ with respect to the Gale ordering $\leq_{\omega}$.
\end{definition}
If $\mathcal{F}$ is a flag matroid, the collection $\{F^i | F \in \mathcal{F} \}$ is called the \emph{$i$-th constituent} of $\mathcal{F}$. This is clearly a matroid (of rank $k_i$).
\begin{remark}
	In the literature it is usally required that we have strict inequalities $0 < k_1 < \ldots < k_s < n$. From a combinatorial point of view this does not make a difference, but when we later consider flag matroid polytopes this restriction would appear artificial.
\end{remark}
Next, we want to describe which tuples of matroids can arise as the constituents of a flag matroid. In order to give this characterisation, we first need to recall \emph{matroid quotients}.
\subsection{Matroid quotients}
\begin{definition} \label{def:quotient}
	Let $M$ and $N$ be matroids on the same ground set $E$. We say that $N$ is a \emph{quotient} of $M$ if one of the following equivalent statements holds:
	\begin{enumerate}[label=(\roman*)]
		\item \label{item:circuit} Every circuit of $M$ is a union of circuits of $N$.
		\item \label{item:rank} If $X\subseteq Y \subseteq E$, then $r_M(Y)-r_M(X) \geq r_N(Y)-r_N(X)$.
		\item \label{item:bigmatroid} There exists a matroid $R$ and a subset $X$ of $E(R)$ such that $M=R\backslash X$ and $N=R/X$. 
		\item For all bases $B$ of $M$, $x \notin B$, there is a basis $B'$ of $N$ with $B' \subseteq B$ and such that $\{y: (B'-y)\cup x \in \mathcal{B}(N)\} \subseteq \{y: (B-y)\cup x \in \mathcal{B}(M)\}$. 
	\end{enumerate}
\end{definition}
For the equivalence of \ref{item:circuit}, \ref{item:rank} and \ref{item:bigmatroid}, we refer to \cite[Proposition 7.3.6]{oxley}.

Here are some basic properties of matroid quotients:
\begin{prop}
	Let $N$ be a quotient of $M$.
	\begin{enumerate}[label=(\roman*)]
		\item Every basis of $N$ is contained in a basis of $M$, and every basis of $M$ contains a basis of $N$.
		\item $\rk(N) \leq \rk(M)$ and in case of equality $N=M$.
	\end{enumerate}
\end{prop}
\begin{proof}
 Both statements can be easily deduced by plugging in $Y=E$ or $X=\emptyset$ in Definition \ref{def:quotient} \ref{item:rank}.
\end{proof}
The next result will be essential for defining representable flag matroids. It also explains where the term ``matroid quotient" comes from --
 below we think of $W$ as a vector space quotient of $V$.
\begin{prop}[\kern-0.5em {\cite[Proposition 7.4.8 (2)]{white}}] \label{prop:realconcor}
Let $V$ and $W$ be vector spaces and $\psi: E \to V$ be a map. Furthermore, let $f:V \to W$ be a linear map. Consider the matroid $M$ represented by $\psi$, and the matroid $N$ represented by $f \circ \psi$. Then $N$ is a matroid quotient of $M$.
\end{prop}
\begin{example}
	If $R$ is a representable matroid on $E$ and $X$ is a subset of $E$, then $M:=R\backslash X$ and $N:=R/X$ are representable matroids, and there is a linear map as in Proposition \ref{prop:realconcor}. Indeed, if $R$ is represented by $\psi: E\to V$, then consider the projection $\pi: V \to V/{\langle \psi(X) \rangle}$. It is not hard to see that $M$ is represented by $\restr{\psi}{E-X}$ and that $N$ is represented by $\pi \circ \restr{\psi}{E-X}$.
\end{example}
\begin{remark}
	The converse of Proposition \ref{prop:realconcor} is false: we now give an example (taken from \cite[Section 1.7.5]{borovik}) of two representable matroids $M$ and $N$ such that $N$ is a quotient of $M$, but there is no map as in Proposition \ref{prop:realconcor}.
\end{remark}
\begin{example} \label{eg:nonrepflag}
	Let $M$ be the rank-3 matroid on $[8]$ represented by the following matrix
	\[
	\begin{pmatrix}
	1 & 0 & 1 & 0 & 1 & 1 & 0 & 1 \\
	0 & 1 & 1 & 0 & 2 & 2 & 2 & 1 \\
	0 & 0 & 0 & 1 & 1 & 2 & 1 & 1
	\end{pmatrix}
	\]
	and let $N$ be the rank-2 matroid on $[8]$ whose bases are all $2$-element subsets except for $\{2,6\}$ and $\{3,5\}$. It is easy to see that $N$ is a representable matroid: just pick six pairwise independent vectors in the plane,
	and map $2$ and $6$, as well as $3$ and $5$, to the same vector.
	Now $N$ is a matroid quotient of $M$, since the matroid $R$ from Example \ref{eg:nonpappus} satisfies $M=R\backslash 9$ and $N=R/9$. 
	However, it is not possible to find representations $V$ (resp.\ $W$) of $M$ (resp.\ $N$) such that there is a map $f:V \to W$ as in Proposition \ref{prop:realconcor}. Roughly speaking, the problem is that the ``big" matroid $R$ is not representable. For a more precise argument, see \cite[Section 1.7.5]{borovik}.
\end{example}
\subsection{Representable flag matroids}
We now come to the promised characterisation of constituents of flag matroids. In fact, it will turn out we can use it as an alternative definition of flag matroids.\\
We call a collection $(M_1,\ldots,M_s)$ of matroids \emph{concordant} if for every pair $(M_i, M_j)$, either $M_i$ is a quotient of $M_j$ or vice versa. Note that this is equivalent to the fact that they can be ordered in such a way that $M_i$ is a quotient of $M_{i+1}$.
\begin{theorem}[\kern-0.5em {\cite[Theorem 1.7.1]{borovik}}]
	A collection $\mathcal{F}$ of flags in $\mathcal{F}^{\bk}_E$ is a flag matroid if and only if the following three conditions hold:
	\begin{enumerate}
		\item Every constituent $M_i := \{F^i | F \in \mathcal{F} \}$ is a matroid.
		\item The matroids $M_1,\ldots,M_s$ are concordant.
		\item Every flag $B_1 \subseteq \ldots \subseteq B_s$, with $B_i$ a basis of $M_i$, is in $\mathcal{F}$.
	\end{enumerate}
\end{theorem}
In other words, flag matroids on $E$ are in one-to-one correspondence with tuples of concordant matroids on $E$.

We can now define representable flag matroids:
let $0 \subsetneq V_1 \subseteq \cdots \subseteq V_s \subsetneq \K^n$ be a flag of subspaces of $\K^n$. Then, viewing $V_i$ as a point in $G(k_i,n)$, it holds that $(M_{V_1},\ldots,M_{V_s})$ is a concordant collection of matroids by Proposition \ref{prop:realconcor}. Indeed: if $V_i \subseteq V_j$, we can pick a $k_j \times n$ matrix $A_j$ representing $V_j$ such that the first $k_i$ rows of $A_j$ span $V_i$. Then the columns of $A_i$ are obtained from the columns of $A_j$ by deleting the last $k_j-k_i$ entries. 
\begin{definition} \label{def:repFlagMatroid}
The \emph{representable flag matroid} $\mathcal{F}(V_1 \subseteq \cdots \subseteq V_s)$ is the unique flag matroid whose constituents are $M_{V_1},\ldots,M_{V_s}$.
\end{definition}
\begin{remark}
	Example \ref{eg:nonrepflag} shows that it can happen that all constituents of a flag matroid are representable matroids, but still the flag matroid is not representable (because the matroid representations are ``not compatible").
\end{remark}
\subsection{Flag matroid polytopes}
\begin{definition}
Given a flag $F$ on $[n]$, we can identify each constituent with a $0,1$-vector and then add them all up to a vector $\be_F \in \Z+^n$. In this way we have identified the flags in $[n]$ of rank $\bk$ with integer vectors with $k_1$ entries equal to $m$, $k_2-k_1$ entries equal to $m-1$, \ldots, $k_{i+1}-k_i$ entries equal to $m-i$, \ldots, $k_s-k_{s-1}$ entries equal to $1$, and $n-k_s$ entries equal to 0. We will refer to such vectors as \emph{rank-$\bk$ vectors}. 
In other words, if we think of $\bk$ as a partition of length $s$, we can consider the conjugate partition $\bk^*$ of length $\leq n$. Then a rank-$\bk$ vector is a vector $v \in \Z+^n$  obtained from $\bk^*$ by adding $0$'s and permuting the entries.
\end{definition}
\begin{definition}  \label{def:flagMatroidPolytope}
	The \emph{base polytope} of a flag matroid $\mathcal{F}$ on $[n]$ is the convex hull of the set $\{\be_F \ | \ F\in \mathcal{F}\} \subset \RR^n$.
\end{definition}
\begin{example} \label{eg:flagmatroid}
	Let $\mathcal{F}$ be the rank $(1,2)$ flag matroid on $[3]$ whose bases are $1 \subseteq 12$, $1 \subseteq 13$, $2 \subseteq 12$ and $3 \subseteq 13$. Then its base polytope is the convex hull of the points $(2,1,0),(2,0,1),(1,2,0),(1,0,2)$. Its constituents are the uniform rank $1$ matroid on $[3]$, and the rank $2$ matroid with bases $12$ and $13$. The base polytope of this flag matroid is depicted in Figure \ref{fig:flagMatroidPolytope}. $\mathcal{F}$ is a representable flag matroid: a representing flag is for example $\langle \be_1 + \be_2 + \be_3 \rangle \subset \langle \be_1, \be_2 + \be_3 \rangle \subset \K^3$.
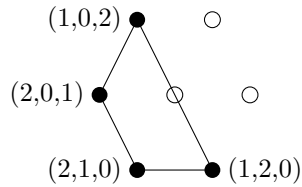
\begin{figure}[h]
	\begin{tikzpicture}
	\coordinate (a) at (0,0);
	\coordinate (b) at (1,0);
	\coordinate (c) at (-0.5,1);
	\coordinate (d) at (0.5,1);
	\coordinate (e) at (1.5,1);
	\coordinate (f) at (0,2);
	\coordinate (g) at (1,2);
	\node (1) at (-0.7,0) {(2,1,0)};
	\node (2) at (-1.2,1) {(2,0,1)};
    \node (3) at (-0.7,2) {(1,0,2)};
    \node (4) at (1.7,0) {(1,2,0)};
	\draw (a) -- (c) -- (f)--(b)--(a);
	\foreach \pt in {a,b,c,f} \fill[black] (\pt) circle (3pt);
	\foreach \pt in {e,d,g}  \draw (\pt) circle (3pt);
	
	\end{tikzpicture}
		\caption{A flag matroid base polytope} \label{fig:flagMatroidPolytope}
\end{figure}
\end{example}

\begin{theorem}[\kern-0.5em {\cite[Theorem 1.11.1]{borovik}}] \label{thm:flagMatroidBasePolytope}
	A lattice polytope $P \subset \RR^n$ is the base polytope of a rank-$\bk$ flag matroid on $[n]$ if and only if the following two conditions hold:
		\begin{enumerate}
		\item Every vertex of $P$ is a rank-$\bk$ vector.
		\item Every edge of $P$ parallel to $\be_i-\be_j$ for some $i,j \in [n]$.
	\end{enumerate}
\end{theorem}
\begin{theorem}[\kern-0.5em {\cite[Corollary 1.13.5.]{borovik}}]
	The polytope of a flag matroid is the Minkowski sum of the matroid polytopes of its constituent matroids.
\end{theorem}
Thus, each flag matroid defines a polytope that is a  base polytope of a polymatroid.
\begin{remark} \label{rmk:interiorFlag}
	It follows from Theorem \ref{thm:sumofpoly} and the previous theorem that the lattice points of a flag matroid polytope correspond to tuples (not necessarily flags) $(B_1,\ldots, B_s)$, where $B_i$ is a basis of $M_i$. For example point $(1,1,1)$ in Figure \ref{fig:flagMatroidPolytope} corresponds to a basis of a polymatroid, but not to a flag, i.e.\ not to a basis of the flag matroid.
\end{remark}
\subsection{Flag matroids and torus orbits}
Consider the flag variety $Fl(\bk,n)$, as described in Section \ref{sec:flagVar}.
The action of the torus $T=(\K^*)^n$ on $\K^n$ induces an action of $T$ on $Fl(\bk,n)$.
A point $p \in Fl(\bk,n)$ gives rise to a representable flag matroid $M$ on $[n]$, as in Definition \ref{def:repFlagMatroid}. All points in the orbit $Tp$ give rise to the same flag matroid. This last statement follows easily from the analogous fact for matroids and the fact that a flag matroid is determined by its constituent matroids.
The analogue of Theorem \ref{thm:PM} holds: 
\begin{theorem}
	 The lattice polytope representing the toric variety $\overline{Tp}$ is isomorphic to the flag matroid polytope of $M$.
\end{theorem}
\begin{proof}
	The proof is a straightforward generalisation of the proof of Theorem \ref{thm:PM}, with
	the parameterisation of $\overline{Tp}$ given by:
	\[\phi:T\rightarrow \PP(\bigwedge^{k_1} \K^n) \times \cdots \times \PP(\bigwedge^{k_s} \K^n).\]
\end{proof}
\section{Representable polymatroids} \label{sec:repPoly}
Representable polymatroids generalise representable matroids, replacing vectors with subspaces.  While they do not appear as frequently in pure mathematics as matroids, one of the possible references for an interested reader is \cite[Section 3]{branden2011obstructions}.

We start by giving a precise definition.
\begin{definition}
	Let $V$ be a vector space and denote by $S(V)$ the set of all subspaces of $V$. Suppose we have a map $\phi:E \to S(V)$ and assume, without loss of generality, that $\sum_{e \in E} \phi(e) = V$. The \emph{representable polymatroid} $\pM(\phi)$ is defined by the rank function: for $A \subseteq E$, define $r(A)$ as the dimension of  $\phi(A):= \sum_{a \in A} \phi(a)$.

\end{definition}
\begin{example} \label{eg:reppolymatroid}
	Consider the map $\phi:[3] \to S(\CC^3)$ defined by $\phi(1)=\langle \be_1,\be_2 \rangle$ and $\phi(2)=\phi(3)=\langle \be_1,\be_3 \rangle$. Then $\pM(\phi)$ is a rank $3$ polymatroid on $[3]$ with 5 bases: $(2,1,0),(2,0,1),(1,2,0),(1,0,2),(1,1,1)$. Its base polytope is on Figure \ref{fig:flagMatroidPolytope}.
\end{example}
Now we want to define the polymatroid analogue of a matroid associated to a subspace. For this, we fix a number $r \in \ZZ_{>0}$, and we consider the vector space $\K^{rn}$, with a fixed basis indexed by $[r]\times [n]$. If we consider a full rank $k \times rn$ matrix $A$, it represents a $k$-dimensional subspace of $\K^{rn}$. Moreover, the $r$ columns indexed by $(i,j)$ for some fixed $j$ span a subspace $W_j$ of $\K^k$. Then the map $[n]\to S(\K^k):j \to W_j$ defines a representable polymatroid. Note that this is an $r$-polymatroid, and that every representable $r$-polymatroid can be obtained in this way.

As before, applying elementary row operations to $A$ does not change the polymatroid. Hence, we have described how to associate an $r$-polymatroid $\pM(W)$ to a subspace $W$ of $\K^{rn}$. However, we cannot arbitrarily rescale the columns of $A$: if we want to keep our subspaces $W_j$, the factor with which we rescale the column $(i,j)$ should not depend on $i$. This suggests that the correct torus action to consider is the action of $T=(\K^*)^n$ on $\K^{rn}$, where $\bt=(t_1,\ldots,t_n)$ acts by multiplying the $(i,j)$-th coordinate by $t_j$.
\begin{remark}
If we defined Grassmannians via quotients instead of subspaces, we could present this story more invariantly, as follows:\\
Let $V$ be a vector space of dimension $rn$, with a fixed collection of $n$ independent $r$-dimensional subspaces $V_1,\ldots V_n$. Then any $k$-dimensional quotient $\phi: V \to W$ gives rise to an $r$-polymatroid via the map
$[n] \to S(W): i \mapsto \phi(V_i)$.
\end{remark}

If we consider the action of $T=(\K^*)^n$ on $G(k,\K^{rn})$ induced by the action on $\K^{rn}$, then every point in the orbit $Tp$ gives the same polymatroid. This is very similar to the situation for representable matroids, so it should be no surprise that we also have an analogue of Theorem \ref{thm:PM} in this setting:
\begin{theorem}
	The lattice polytope representing the projective toric variety $\overline{Tp}$  described above is isomorphic to the base polytope of the polymatroid $\pM$.
\end{theorem}
\begin{proof}
	It is a straightforward generalisation of the proof of Theorem \ref{thm:PM}.
\end{proof}

\subsection{Comparison between polymatroids and flag matroids}
A flag matroid polytope is a special case of a polymatroid polytope, by Theorem  \ref{thm:flagMatroidBasePolytope}. It is tempting to think about flag matroids as a special case of polymatroids, but we need to be careful when doing this: the notion of a basis of a flag matroid is not compatible with the notion of a basis for a polymatroid. More precisely, the bases of a polymatroid are all lattice points of its base polytope, while the bases of a flag matroid are only the vertices of the associated polytope - cf.~Remark \ref{rmk:interiorFlag} and Figure \ref{fig:flagMatroidPolytope}.

For a flag matroid $\mathcal{F}$, let $\pM(\mathcal{F})$ be the associated polymatroid.



The definitions of representable flag matroid and representable polymatroid look unrelated at first sight, but there is a connection.
\begin{prop} \label{prop:repflagpoly}
	If $\mathcal{F}$ is a representable flag matroid, then $\pM(\mathcal{F})$ is a representable polymatroid.
\end{prop}
\begin{proof}
For this proof we will use the more intrinsic definition of representable flag matroids and polymatroids using quotients. Let $\mathcal{F}$ be the flag matroid represented by the flag of quotients $\K^n \to V_1 \to \cdots \to V_r$. Then the quotient $\K^{kn} \to V_1 \bigoplus \cdots \bigoplus V_r$ represents a polymatroid $\pM$ on $[n]$. We need to argue that $\mathcal{F}$ and $\pM$ have the same base polytope.

Bases of $\mathcal{F}$ correspond to flags $[n] \supsetneq F^1 \supseteq \ldots \supseteq F^r$ such that $F^i$ gives a basis of $V_i$. On the other hand, choosing a basis of $\pM$ corresponds to choosing for every $i$ a subset $F^i \subsetneq [n]$ such that $F^i$ gives a basis of $V_i$. From this it follows that every vertex of $P(\mathcal{F})$ is a lattice point of $P(\pM)$, and (using Remark \ref{rmk:interiorFlag}) that every lattice point of $P(\pM)$ is a lattice point of $P(\mathcal{F})$. So $P(\pM)=P(\mathcal{F})$ as desired.
\end{proof}
\begin{remark}
Thinking again about flag varieties in terms of subspaces, we have that if $\mathcal{F}$ is represented by a flag $V_1 \subseteq \ldots \subseteq V_r \subsetneq \K^n$, then  $\pM(\mathcal{F})$ is represented by a subspace $V_1 \oplus \cdots \oplus V_r \subsetneq \K^{rn}$.
Geometrically, this construction corresponds to an algebraic map
\[
Fl(k_1,\ldots, k_r; n) \to G(\sum_{i}{k_i},rn) \text{.}
\]	
\end{remark}
\begin{remark}
The converse of Proposition \ref{prop:repflagpoly} is not true: given a representable polymatroid that is also a flag matroid, it is not always a representable flag matroid. One way to construct a counterexample is as follows. If $M_1$, $M_2$ are representable matroids with corresponding base polytopes $P_1$, $P_2$, then $P_1+P_2$ corresponds to a representable polymatroid. However, Example \ref{eg:nonrepflag} gives an example of two such (concordant) matroids such that $P_1+P_2$ corresponds to a flag matroid that is not representable.
\end{remark}

\section{Equivariant $K$-theory}
We have presented a correspondence between representable matroids and torus orbits in Grassmannians, and generalisations of this correspondence to representable flag matroids and representable polymatroids. We would like to drop the word ``representable" from all of those. As we will see, in order to do this, we need to replace ``torus orbits" with ``classes in equivariant $K$-theory". This was done for matroids by Fink and Speyer \cite{FinkSpeyer}. In this section, we review their construction, and consider generalisations to flag matroids. 
 We start with an introduction to non-equivariant and equivariant $K$-theory.
\subsection{A very brief introduction to $K$-theory}
This section is based on \cite[Section 15.1]{FultonIntersection}.

%

Let $X$ be an algebraic variety. We define $K^0(X)$ to be the free abelian group generated by vector bundles on $X$, subject to relations $[E]=[E']+[E'']$ whenever $E'$ is a subbundle of $E$, with quotient bundle $E''=E/E'$. The group $K^0(X)$ inherits a ring structure from the tensor product: $[E]\cdot [F]=[E \otimes F]$.

Similarily, we can define $K_0(X)$ to be the free abelian group generated by isomorphism classes of coherent sheaves on $X$, subject to relations $[A]+[C]=[B]$ whenever there is a short exact sequence $0 \to A \to B \to C \to 0$. There is an inclusion $K^0(X) \hookrightarrow K_0(X)$. From now on, we will always assume that $X$ is a smooth variety. In this case, the inclusion is an isomorphism, allowing us to identify $K^0(X)$ and $K_0(X)$.

Let $f: X \to Y$ be a map of (smooth) varieties. Then there is a pullback map $f^*: K^0(Y) \to K^0(X)$ defined by $f^*[E]=[f^*E]$ (where E is a vector bundle on $Y$). If $f$ is a proper map, there is also a pushforward map $f_*: K_0(X) \to K_0(Y)$ given by $f_*[A] = \sum{(-1)^i[R^if_*A]}$. Here $R^if_*$ are right derived functors of the push-forward. An interested reader is advised to find the details in \cite[Section 15]{FultonIntersection}. In this paper, we will not be using the formal definitions of $K^0(X)$, $f^*$ or $f_*$. Instead, we will refer to explicit descriptions of those in the cases that we need, each time providing a theorem we build upon.


\begin{remark}
In all the cases we study the ring $K^0(X)$ is isomorphic to the cohomology ring and to the Chow ring (after tensoring with $\QQ$). Note however that the map from $K^0(X)$ to the Chow ring is nontrivial and given by the Chern character.
\end{remark}
\begin{example}
Consider the projective space $\PP^n$.  The (rational) Chow ring is $A(\PP^n)=\QQ[H]/(H^{n+1})$. Here one should think about $H$ as a hyperplane in $\PP^n$ and $H^k$ as codimension $k$ projective subspace. The most important line bundle is $\mathcal{O}(1)$. The Chern character $ch:K^0(\PP^n) \to A(\PP^n)$ sends $[\mathcal{O}(1)]$ to $\sum_{i=0}^n H^i/i!$. Note that $K^0(\PP^n)$ can be written as $\ZZ[\alpha]/(\alpha^{n+1})$, where $\alpha = 1-[\mathcal{O}(-1)]$ is the class of the structure sheaf of a hyperplane. As a special case, the $K$-theory of a point is $\ZZ$.
\end{example}




If $X$ is a smooth variety equipped with an action of a torus $T$, we can define its \emph{equivariant $K$-theory} $K_T^0(X)\cong K^T_0(X)$. The construction and properties are exactly the same as in the previous section, if we replace ``vector bundles" and ``coherent sheaves" by ``T-equivariant vector bundles" and ``T-equivariant coherent sheaves".

For later reference, we describe the equivariant $K$-theory of a point: $K_T^0(pt)=\ZZ[\Char(T)]$, where $\Char(T)=\Hom(T,\K^*)$ is the lattice of characters of $T$. Here $\ZZ[\Char(T)]$ is a ring, which as a module over $\ZZ$ has basis given by $\Char(T)$, and multiplication is induced from addition in $\Char(T)$. It is the ring of Laurent polynomials in $\dim T$ variables.
Explicitly, a $T$-equivariant sheaf on $pt$ is just a vector space $W$ with a $T$-action. We may decompose $W=\oplus_{{\bf c}\in \Char(T)} W_{\bf c}$ as in Section \ref{sub:repsandchar}. The corresponding element of $\ZZ[\Char(T)]$ is the character (also called \emph{Hilbert series}) $\Hilb(W) := \sum_{{\bf c}\in \Char(T)} (\dim W_{\bf c}){\bf c}$. We point out that even for infinite-dimensional $T$-modules, $\Hilb(W)$ makes sense as a formal power series, as long as $W_{\bf c}$ is finite-dimensional for all $\bf c$.

We finish this section by describing the relation between ordinary and $T$-equivariant $K$-theory:
\begin{theorem}[{\kern-0.5em \cite[Theorem 4.3]{merkurev}}] \label{thm:smallTorus}
	Let $X$ be a smooth projective variety with an action of a torus $T$. Let $S \subseteq T$ be a subtorus.
	Then the natural map
	\[
	K^0_T(X) \otimes_{\ZZ[\Char(T)]} \ZZ[\Char(S)] \to K^0_S(X)
	\]
	is an isomorphism. In particular, taking $S$ to be the trivial group, the natural map
	\[
	K^0_T(X) \otimes_{\ZZ[\Char(T)]} \ZZ \to K^0(X)
	\]
	is an isomorphism.
\end{theorem}
\subsection{Explicit construction via equivariant localisation}
Let $X$ be a smooth projective variety over $\K$, and $T$ a torus acting on it. If $X$ has only finitely many torus-fixed points, we can use the method of \emph{equivariant localisation} to give an explicit combinatorial description of classes in $K^0_T(X)$. Our exposition here is largely based on the one in \cite{FinkSpeyer}.
The following theorem is central to our discussion.
\begin{theorem}[{\kern-0.5em {\cite[Theorem 3.2]{nielsen}, see also \cite[Theorem 2.5]{FinkSpeyer} and the references therein}}] \label{thm:equiloc}
	If $X$ is a smooth projective variety with a torus action, then the restriction map $K^0_T(X) \to K^0_T(X^T)$ is an injection.
\end{theorem}
From now on we will always assume that $X$ has only finitely many torus-fixed points. In this case $K^0_T(X^T)$ is simply the ring of functions from $X^T$ to $\ZZ[\Char(T)]$. In other words, we can describe a class in $K^0_T(X)$ just by giving a finite collection of Laurent polynomials in $\ZZ[\Char(T)]$.
\begin{remark}
	In the literature, a variety $X$ for which $K^0_T(X)$ is a free $\ZZ[\Char(T)]$-module, and has a $\ZZ[\Char(T)]$-basis that restricts to a $\ZZ$-basis of $K^0(X)$, is called \emph{equivariantly formal}. This notion was first introduced in \cite{GKM}.
	In \cite[Section 2.4]{anderson}, it is noted that smooth projective varieties with finitely many $T$-fixed points are equivariantly formal. 
\end{remark}
We now explicitly describe the class of a $T$-equivariant coherent sheaf on $X$. We will do this under the following additional assumption (which is not essential but makes notation easier and will hold for all varieties of interest):
\begin{definition} \label{def:contracting}
	A finite-dimensional representation of $T$ is called \emph{contracting} if all characters lie in an open halfspace; or equivalently, if the characters generate a pointed cone (see Section \ref{subsec:cones}).
	The action of $T$ on a variety $X$ is \emph{contracting}, if for every torus-fixed point $x \in X$, there exists an open neighbourhood $U_x$ isomorphic to $\AA^N$ such that the action of $T$ on $U_x$ is a contracting representation.
\end{definition}
Let $E$ be a $T$-equivariant coherent sheaf on $X$. We will construct a map $[E]^T: X^T \to \ZZ[\Char(T)]$. For every $x \in X^T$, we have an open neighbourhood $U_x$ as in Definition \ref{def:contracting}.
Let $\chi_1,\ldots,\chi_N$ be the characters by which $T$ acts on $U_x$ (so $\mathcal{O}(U_x)$ is a polynomial ring multigraded by $T$ in the sense of \cite[Definition 8.1]{MillerSturmfels}, with characters $\chi_1^{-1},\ldots,\chi_N^{-1}$).
Our sheaf $E$, restricted to $U_x$, corresponds to a graded, finitely generated $\mathcal{O}(U_x)$-module $E(U_x)$.

Since $E(U_x)$ is a graded module over the polynomial ring $\mathcal{O}(U_x)$, which is multigraded by $T$, it follows from \cite[Theorem 8.20]{MillerSturmfels} that $E(U_x)$ is a $T$-module, and its Hilbert series is of the form
\begin{equation} \label{eqn:KclassE}
\frac{K(E(U_x),\bt)}{\prod_{i=1}^{N}{(1-\chi_i^{-1})}} \text{,}
\end{equation}

for some $K(E(U_x),\bt) \in \ZZ[\Char(T)]$.
\begin{definition}
	For $E$ a $T$-equivariant coherent sheaf on $X$, we define $[E]^T$ to be the map that sends $x \in X^T$ to $K(E(U_x),\bt) \in \ZZ[\Char(T)]$, the numerator in (\ref{eqn:KclassE}).
\end{definition}
\begin{theorem}[{\kern-0.5em \cite[Theorem 2.6]{FinkSpeyer}}]
	The map $[E]^T$ defined above is the image of the class of $E$ under the injection $K^0_T(X) \hookrightarrow K^0_T(X^T)$ of Theorem \ref{thm:equiloc}.
\end{theorem}
\begin{example} \label{eg:O(d)}
	Let $X=\PP^n$, equipped with the natural torus action $\bt\cdot [a_0: \ldots : a_n] = [t_0^{-1}a_0:\ldots :,t_n^{-1}a_n]$. Then $\mathcal{O}(d)$ is a $T$-equivariant sheaf. The torus action on $\PP^n$ has $n+1$ fixed points, namely $p_i=[0:\ldots: 1:\ldots:0]$, where the $1$ is at position $i$. We use equivariant localisation to describe the class $[\mathcal{O}(d)]^T$.\\
	Every $p_i$ has an open neighbourhood $U_i=\Spec A_i$, where $A_i= \K[x_0,\ldots,\hat{x_i},\ldots,x_n]$ is multigraded by $T$ via $\deg(x_j)=t_i^{-1}t_j$. The $A_i$-module $\mathcal{O}(d)(U_i)$ is a copy of $A_i$ generated in degree $t_i^{d}$. So its Hilbert series is $t_i^d/{\prod_j{(1-t_i^{-1}t_j)}}$. Hence $[\mathcal{O}(d)]^T$ can be represented by the map $(\PP^n)^T \to \ZZ[\Char(T)]: p_i \mapsto t_i^d$.
\end{example}
We can describe the image of the map from Theorem \ref{thm:equiloc} explicitely, if we impose an additional condition on $X$.
\begin{theorem}[{\kern-0.5em {\cite[Corollary 5.12]{VezzosiVistoli}, see also \cite[Theorem 2.9]{FinkSpeyer} and the references therein}}] \label{thm:equilocim}
	Suppose $X$ is a projective variety with an action of a torus $T$, such that $X$ has finitely many $T$-fixed points and finitely many 1-dimensional $T$-orbits, each of which has closure isomorphic to $\PP^1$. Then a map $f:X^T \to \ZZ[\Char(T)]$ is in the image of the map $K^0_T(X) \to K^0_T(X^T)$ of Theorem \ref{thm:equiloc} if and only if the following condition holds:

	For every one-dimensional orbit, on which $T$ acts by character $\chi$ and for which $x$ and $y$ are the $T$-fixed points in the orbit closure, we have
	\[
	f(x) \equiv f(y) \mod{1-\chi} \text{.}
	\]
\end{theorem}
\begin{example}
	We continue Example \ref{eg:O(d)}. Note that $\PP^n$ has only finitely many one-dimensional torus orbits: for every pair $p_i,p_j$ of $T$-fixed points, there is a unique $T$-orbit whose closure contains $p_i$ and $p_j$. Furthermore, $T$ acts on this orbit with character $t_j^{-1}t_i$. We see that $t_i^d \equiv t_j^d \mod{1-t_j^{-1}t_i}$, so that the class $[\mathcal{O}(d)]^T$ indeed fulfills the condition of Theorem \ref{thm:equilocim}.
\end{example}

We also can describe pullback and pushforward in the language of equivariant localisation. Let $\pi: X \to Y$ be a $T$-equivariant map of smooth projective varieties with finitely many $T$-fixed points, then for $[E]^T \in K_T^0(Y)$, its pullback can be computed by
\begin{equation} \label{eqn:pullback}
(\pi^*[E]^T)(x) = [E]^T(\pi(x))
\end{equation}
for $x \in X^T$.

Describing the pushforward of $[F]^T \in K^0_T(X)$ is a bit more complicated. Suppose that $X$ and $Y$ are contracting. For every point $x \in X^T$, (resp.\ $y \in Y^T$) we pick as before an open neighbourhood $U_x$ (resp.\ $V_y$) on which $T$ acts by characters $\chi_1(x),\ldots, \chi_r(x)$ (resp.\ $\eta_1(y),\ldots, \eta_s(y)$). Then the pushforward of $[F]^T$ is determined by the formula
\begin{equation} \label{eqn:pushforward}
\frac{(\pi_*[F]^T)(y)}{\prod{(1-\eta_j(y)^{-1})}} = \sum_{x \in \pi^{-1}(y) \cap X^T}{\frac{[F]^T(x)}{\prod{(1-\chi_i(x)^{-1})}}} \text{.}
\end{equation}
\begin{remark}
	We can use Theorem \ref{thm:smallTorus} to obtain a description of the ordinary $K$-theory using equivariant localisation. However, one should be careful when using this for computations in practice. Here is a toy example: let $X=\PP^2$ with the usual action of $(\CC^*)^2$. Then $X^T=\{[1:0],[0:1]\}$, and we can write the elements of $K^0_T(X^T) \cong \Maps(X^T,\ZZ[t_0^{\pm},t_1^{\pm}]) \ni f$ as pairs $(f([1:0]),f([0:1]))$. Then $(t_0-t_1,0)$ satisfies the condition from Theorem \ref{thm:equilocim}, hence it gives a class in $K^0_T(X)$. It is tempting to do the following computation in $K^0(X) \cong K^0_T(X) \otimes_{\ZZ[\Char(T)]} \ZZ$:
	\[
	(t_0-t_1,0) \otimes 1 = (1,0) \otimes (1-1) = 0
	\]
	but this is wrong! Indeed, $(1,0)$ does not satisfy the condition from Theorem \ref{thm:equilocim}, hence is not in $K^0_T(X)$. In fact, one can check that $(t_0-t_1,0)$ is the equivariant class of the torus-fixed point $[1:0] \in \PP^2$.
\end{remark}
\subsection{A short review on cones and their Hilbert series} \label{subsec:cones}
For more details about the topic of this subsection we refer to \cite[Section 1.2]{Cox} and \cite[Section 4.5]{Stanley}.
Recall that a \emph{convex polyhedral rational cone} is a subset of $\RR^n$ of the form $C=\Cone(S):=\{\sum_{\bu \in S}{\lambda_{\bu}\bu}|\lambda_{\bu} \in \R+\}$, where $S \subset \ZZ^n \subset \RR^n$ is finite. A cone is called \emph{pointed} if it does not contain a line. If $C$ is a pointed rational cone, then every one-dimensional face $\rho$ contains a unique lattice point $\bu_{\rho}$ that is closest to the origin. It is not hard to see that $MG(C):=\{\bu_{\rho}|\rho \text{ a one-dimensional face of } C \}$ is a minimal generating set of $C$. If the minimal generators are linearly independent over $\RR$, we call $C$ \emph{simplicial}. If they are part of a $\ZZ$-basis of $\ZZ^n$, we call $C$ \emph{regular}.

For a pointed cone $C$ in $\RR^n$, we define its \emph{Hilbert series} $\Hilb(C)$ by:
\[
\Hilb(C) := \sum_{\ba \in C \cap \ZZ^n}{\bt^\ba} \text{.}
\]
This is always a rational function, whose denominator is equal to $\prod_{\bu \in MG(C)}{(1-\bt^{\bu})}$ \cite[Theorem 4.5.11]{Stanley}.

If $C$ is a regular cone, then its Hilbert series is easy to compute: $\Hilb(C)=\prod_{\bu \in MG(C)}{\frac{1}{1-\bt^{\bu}}}$.

If $C$ is a simplicial cone, we can compute its Hilbert series as follows. First compute the finite set $D_C:=\{\bb \in C \cap \ZZ^n : \bb = \sum_{\bb \in MG(C)}{\lambda_\bu\bu}| 0 \leq \lambda_u < 1\}$. Then
\[
\Hilb(C)=(\sum_{\bb \in D_C}{\bt^\bb})\prod_{\bu \in MG(C)}{\frac{1}{1-\bt^{\bu}}}.
\]

For a general rational polydral cone, we can compute its Hilbert series by triangulating it.
\subsection{Matroids and the $K$-theory of Grassmannians}
In this subsection we compute the class in equivariant $K$-theory of a torus orbit closure in a Grassmannian. We then note that this class only depends on the underlying matroid, and give a combinatorial algorithm to get the class in $K$-theory directly from the matroid. This algorithm can then be used as a definition to associate a class in $K$-theory to an arbitrary (not necessarily representable) matroid. This was first done by Fink and Speyer in \cite{FinkSpeyer}.

Let us first fix the following sign conventions. The torus $T=(\CC^*)^n$ acts on $\CC^n$ as follows: $\bt \cdot (x_1,\ldots,x_n) = (t_1^{-1}x_1,\ldots,t_n^{-1}x_n)$. The action of $T$ on $G(k,n)$ is induced from this action. Explicitly, if $p \in G(k,n)$ has Pl\"ucker coordinates $[P_I]_{I\in \binom{[n]}{k}}$, then $\bt \cdot p$ has Pl\"ucker coordinates $[(\prod_{i \in I}{t_i^{-1}})P_I]_{I\in \binom{[n]}{k}}$.

We begin by describing the $T$-equivariant $K$-theory of the Grassmannian $G(k,n)$ using equivariant localisation.

The torus-fixed points of $G(k,n)$ are easy to describe: for every size $k$ subset $I \subset [n]$, we define the $k$-plane $V_I = \Span(\{e_i | i \in I\}) \subset \K^n$, and denote the corresponding point in $G(k,n)$ by $p_I$. In Pl\"ucker coordinates, $p_I$ is given by $P_J=0$ if $J\neq I$. It is easy to see that the $\binom{n}{k}$ points $p_I$ are precisely the torus-fixed points of $G(k,n)$.

We can also describe the one-dimensional torus orbits: there is a (unique) one-dimensional torus orbit between $p_I$ and $p_J$ if and only if $|I \cap J| = k-1$. In this case, we write $I-J=\{i\}$, $J-I=\{j\}$. If we identify the one-dimensional orbit from $p_I$ to $p_J$ with $\AA^1\setminus{0}$ in such a way that the origin corresponds to the torus-fixed point $p_I$ (and so $p_J$ corresponds to the point at infinity), then $T$ acts on the orbit with character $t_j^{-1}t_i$.

Let us now check that the action of $T$ is contracting. We fix a torus-fixed point $p_I$, and  consider the open neighbourhood $U_I$ given by $P_{I} \neq 0$. Then $U_I \cong \AA^{k(n-k)}$. For $p \in U_I$, we will denote its coordinates with $(u_{i,j})_{i \in I, j \notin I}$, where $u_{i,j}=\frac{P_{I-i \cup j}}{P_I}$.  Then $\bt \cdot p$ has coordinates $(t_j^{-1}t_i u_{i,j})_{i \in I, j \notin I}$. Thus, $T$ acts on this space with characters $t_j^{-1}t_i$, where $i \in I, j \notin I$. Identifying $t_1^{a_1}\cdots t_n^{a_n}$ with $(a_1,\ldots,a_n)$, all these points lie in the open halfspace defined by $\sum_{i \in I}{a_i} > 0$. 


\begin{example}
	We compute the class of $\mathcal{O}(1)$. The sheaf $\mathcal{O}(1)$ on $G(k,n)$ is the pullback of $\mathcal{O}(1)$ on $\PP^{\binom{n}{k}-1}$ via the Pl\"ucker embedding. We can also describe $\mathcal{O}(1)$ as $\bigwedge^k S^{\vee}$, where $S$ is the tautological vector bundle on $G(k,n)$ whose fiber over a point is the corresponding $k$-plane.\\
	Using the result from Example \ref{eg:O(d)} (with a different torus action, induced from the action on the Pl\"ucker coordinates) and the pullback formula (\ref{eqn:pullback}), we find that the class $[\mathcal{O}(1)]^T$ in equivariant $K$-theory is the map
	\[
	[\mathcal{O}(1)]^T : Gr(k,n)^T \to \ZZ[\Char(T)]:
	p_I \mapsto t_{i_1}\cdots t_{i_k},
	\]
	where we wrote $I = \{i_1, \ldots, i_k\}$.
\end{example}

Let $p$ be a point in $Gr(k,n)$ and $M=M_p$ be the corresponding matroid. Then $\overline{Tp}$ is a closed subvariety of $Gr(k,n)$; in particular, it is given by a coherent sheaf. We want to compute its class in $T$-equivariant $K$-theory, which is a map $[\overline{Tp}]^T: Gr(k,n)^T \to \ZZ[\Char(T)]$. As above, let $p_I \in Gr(k,n)^T$ be the torus-invariant point given by $P_J=0$ for $J \neq I$, and let $U_I$ be the affine open neighbourhood $U_I$ of $p_I$ defined by $P_I = 1$.

If $I$ is not a basis of $M$, then $\overline{Tp}$ does not intersect $U_I$, hence $[\overline{Tp}]^T(p_I)=0$. Hence, we will assume that $I$ is a basis of $M$, i.e.\ that $p \in U_I$.


The coordinate ring of $\overline{Tp} \cap U_I$ is isomorphic to $\K[s_i^{-1}s_j]$, where $s_i^{-1}s_j$ is a generator if and only if $(I - i) \cup j$ is a basis of $M$. We will denote this ring by $R_{M,I}$.
This ring should be viewed as a $T$-module, with $\bt \cdot s_i^{-1}s_j = t_i^{-1}t_j s_i^{-1}s_j$.
The Hilbert series of $R_{M,I}$ is a rational function with denominator dividing $\prod_{i\in I}\prod_{j\notin I}(1-t_i^{-1}t_j)$. Thus, by (\ref{eqn:KclassE}),
\begin{equation} \label{eqn:Kmatroid1}
[\overline{Tp}]^T(p_I) = \Hilb(R_{M,I})\prod_{i\in I}\prod_{j\notin I}(1-t_i^{-1}t_j) \text{.}
\end{equation}
\begin{definition} \label{def:conevp}
	For any lattice polytope $P$ and $v$ a vertex of $P$, we define $\Cone_v(P)$ to be the cone spanned by all vectors of the form $u-v$ with $u \in P$. For $I \in \binom{[n]}{k}$, we write $\Cone_I(M) := \Cone_{\be_I}(P(M))$ if $I$ is a basis of $M$, and $\Cone_I(M) := \emptyset$ otherwise.
\end{definition}
Since  $\Cone_I(M)$ is the positive real span of all vectors $e_J-e_I$, where $J \in \mathcal{B}(M)$, we find that $\Hilb(R_{M,I}) =  \Hilb(\Cone_I(M))$.
So (\ref{eqn:Kmatroid1}) can also be written as
\begin{equation} \label{eqn:Kmatroid2}
[\overline{Tp}]^T(p_I) = \Hilb(\Cone_I(M))\prod_{i\in I}\prod_{j\notin I}(1-t_i^{-1}t_j) \text{.}
\end{equation}
We note that (\ref{eqn:Kmatroid2}) only depends on the matroid $M$ and not on the chosen point $p$ or even the torus orbit $\overline{Tp}$. Moreover, the formulas make sense even for non-representable matroids. Thus we can use them as a definition for the class in $K$-theory for a matroid:
\begin{definition}[{\kern-0.5em \cite{FinkSpeyer}}]
	For any rank $k$ matroid $M$ on $[n]$, we define $y(M): Gr(k,n)^T \to \ZZ[\Char(T)]$ by
	\[
	y(M)(p_I) = \Hilb(\Cone_I(M))\prod_{i\in I}\prod_{j\notin I}(1-t_i^{-1}t_j) \text{.}
	\]
\end{definition}
\begin{theorem}[{\kern-0.5em \cite[Proposition 3.3]{FinkSpeyer}}]
	The class $y(M) \in K^0_T(Gr(k,n)^T)$ satisfies the condition of Theorem \ref{thm:equilocim}, and hence defines a class in $K^0_T(Gr(k,n))$.
\end{theorem}
\subsection{Flag matroids and the $K$-theory of flag varieties}
In this section, we generalize the results from the previous section replacing ``matroids" by ``flag matroids" and ``Grassmannians" by ``flag varieties".

We first describe the equivariant $K$-theory of a flag variety $Fl(\bk,n)$.
The torus-fixed points are given as follows: for every (set-theoretic) flag $F=(F_1 \subseteq \ldots \subseteq F_s)$ of rank $\bk$ on $[n]$, we define a (vector space) flag $V_1 \subseteq \ldots \subseteq V_s$ by $V_i = \Span(\{e_j | j \in F_i\})$. We will denote the corresponding point in $Fl(\bk,n)$ by $p_F$. The Pl\"ucker coordinates of $p_F$ are given by $P_S=1$ if $S$ is a constituent of $F$ and $P_S=0$ otherwise. Here, the Pl\"ucker coordinates of a point in $Fl(\bk,n)$ are the ones induced from the embedding $Fl(\bk,n) \hookrightarrow \prod{G(k_i,n)}$.

We can also describe the one-dimensional torus orbits: let $p_F$ be a torus-fixed point. We define $S(F)$ to be the set of all pairs $(i,j) \in [n] \times [n]$ for which there exists an $\ell$ such that $i \in F_{\ell}$ and $j \notin F_{\ell}$.  For every $(i,j) \in S(F)$, we define a new flag $F'=F_{i \to j}$ by switching the roles of $i$ and $j$. More precisely: if $i \in F_r$ but $j \notin F_r$, then $F'_r := (F_r - i) \cup j$; in any other case $F'_r := F_r$. Then there is a unique one-dimensional torus orbit between $p_F$ and $p_{F'}$, and all one-dimensional torus orbits arise in this way. 
The torus $T$ acts on this orbit with character $t_j^{-1}t_i$.

\begin{lemma}
	The action of $T$ on $Fl(\bk,n)$ is contracting.
\end{lemma}
\begin{proof}
	For every torus-fixed point $p_F$, we consider the open neighbourhood $U_F$ given by $P_{F_r} \neq 0$ for all $r$. Then $U_F \cong \AA^N$, where $N = \dim(Fl(\bk,n)) = \sum_{i=1}^s{k_i(k_{i+1}-k_i)}$ (here $k_{s+1}:=n$).
	We will denote the coordinates of a point $q$ in $U_F$ by $(u_{i,j})_{(i,j) \in S(F)}$, where $u_{i,j}=\frac{P_{F_{r}-i \cup j}}{P_{F_{r}}}$ for any $r$ which satisfies $i \in F_{r}$ and $j \notin F_{r}$. Then $\bt \cdot q$ has coordinates $(t_j^{-1}t_i u_{i,j})_{(i,j) \in S(F)}$.
	So $T$ acts on $U_F$ with characters $t_j^{-1}t_i$, $(i,j) \in S(F)$. 
	As before, identifying $t_1^{a_1} \cdots t_n^{a_n}$ with $(a_1,\ldots,a_n)$, all these characters lie on the open halfspace $\sum_{r = 1}^{s} \sum_{i \in F_r}{a_i} >0$.
\end{proof}

Let $p$ be a point in $Fl(\bk,n)$, and let $\mathcal{F}$ be the corresponding flag matroid. 
We want to compute the $T$-equivariant class of $\overline{Tp}$, which is a map $[\overline{Tp}]^T: Fl(\bk,n)^T \to \ZZ[\Char(T)]$.
We fix a point $p_F \in Fl(\bk,n)^T$, and consider the affine neighbourhood $U_F$ described above.

If $F$ is not a basis of $\mathcal{F}$, then $\overline{Tp}$ does not intersect $U_F$, hence $[\overline{Tp}]^T(p_F)=0$. Thus, we will assume that $F$ is a basis of $\mathcal{F}$, i.e.\ that $p \in U_F$.

The coordinate ring of $\overline{Tp} \cap U_F$ is isomorphic to $k[s_i^{-1}s_j]$, where $s_i^{-1}s_j$ is a generator if and only if $F_{i \to j} \in \mathcal{F}$.
We will denote this ring by $R_{\mathcal{F},F}$. This ring should be viewed as a $T$-module, with $\bt \cdot s_i^{-1}s_j = t_i^{-1}t_j s_i^{-1}s_j$. The Hilbert series of this $T$-module is a rational function with denominator dividing $\prod_{(i,j) \in S(F)}(1-t_i^{-1}t_j)$. Thus, by (\ref{eqn:KclassE}),
\begin{equation} \label{eq:Kflagmatroid1}
[\overline{Tp}]^T(p_F) = \Hilb(R_{\mathcal{F},F})\prod_{(i,j) \in S(F)}(1-t_i^{-1}t_j) \text{.}
\end{equation}
\begin{definition}
	We define $\Cone_F(\mathcal{F})$ to be the cone $\Cone_{\be_F}(P(\mathcal{F}))$, as in Definition \ref{def:conevp}.
\end{definition}
As before, we find that $\Hilb(R_{\mathcal{F},F}) =  \Hilb(\Cone_F(\mathcal{F}))$.
Hence, (\ref{eq:Kflagmatroid1}) can also be written as
\begin{equation} \label{eq:Kflagmatroid2}
[\overline{Tp}]^T(p_F) = \Hilb(\Cone_F(\mathcal{F}))\prod_{(i,j) \in S(F)}(1-t_i^{-1}t_j) \text{.}
\end{equation}
As before, (\ref{eq:Kflagmatroid2}) only depends on the flag matroid $\mathcal{F}$ and not on the chosen point $p$ or even the torus orbit $\overline{Tp}$. Moreover the formulas make sense even for non-representable flag matroids. Thus we can use them as a definition for the class in $K$-theory for a flag matroid:
\begin{definition} \label{def:Kclassflag}
For any rank $\bk$ flag matroid $\mathcal{F}$ on $[n]$, we define $y(\mathcal{F}): Fl(\bk,n)^T \to \ZZ[\Char(T)]$ by
\[
y(\mathcal{F})(p_F) = \Hilb(\Cone_F(\mathcal{F}))\prod_{(i,j) \in S(F)}(1-t_i^{-1}t_j) \text{.}
\]
\end{definition}
\begin{prop}
	The class $y(\mathcal{F}) \in K^0_T(Fl(\bk,n)^T)$ satisfies the condition of Theorem \ref{thm:equilocim}, and hence defines a class in $K^0_T(Fl(\bk,n))$.
\end{prop}
\begin{proof}
	The proof is a straightforward generalisation of the proof of {\cite[Proposition 3.3]{FinkSpeyer}}.
\end{proof}
\begin{example} \label{eg:flagmatroidclass}
	Let $\mathcal{F}$ be the flag matroid of Example \ref{eg:flagmatroid}. We first compute $y(\mathcal{F})(p_F)$, where $F$ is the flag $2 \subseteq 12$ (so $e_F=(1,2,0)$). From Figure \ref{fig:flagMatroidPolytope}, it is clear that $\Cone_F(\mathcal{F})=\Cone((1,-1,0),(0,-1,1))$, which has Hilbert series $\frac{1}{(1-t_2^{-1}t_1)(1-t_2^{-1}t_3)}$. Furthermore, we have $S(F)=\{(2,1),(2,3),(1,3)\}$. We find that $y(\mathcal{F})(p_F) = 1-t_1^{-1}t_3$. We can do the same for the other torus-fixed points. The result is summarised in Figure \ref{fig:flagMatroidPolytope2}.
	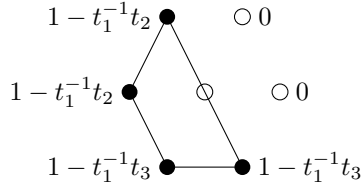
\begin{figure}[h!]
		\begin{tikzpicture}
		\coordinate (a) at (0,0);
		\coordinate (b) at (1,0);
		\coordinate (c) at (-0.5,1);
		\coordinate (d) at (0.5,1);
		\coordinate (e) at (1.5,1);
		\coordinate (f) at (0,2);
		\coordinate (g) at (1,2);
		\node (1) at (-0.9,0) {$1-t_1^{-1}t_3$};
		\node (2) at (-1.4,1) {$1-t_1^{-1}t_2$};
		\node (3) at (-0.9,2) {$1-t_1^{-1}t_2$};
		\node (4) at (1.9,0) {$1-t_1^{-1}t_3$};
		\node (5) at (1.8,1) {0};
		\node (6) at (1.3,2) {0};
		\draw (a) -- (c) -- (f)--(b)--(a);
		\foreach \pt in {a,b,c,f} \fill[black] (\pt) circle (3pt);
		\foreach \pt in {e,d,g}  \draw (\pt) circle (3pt);
		
		\end{tikzpicture}
		\caption{The class in $K$-theory of a flag matroid} \label{fig:flagMatroidPolytope2}
	\end{figure}
\end{example}

\subsection{The Tutte polynomial via $K$-theory} \label{sec:Ktutte}
In \cite{FinkSpeyer}, a geometric description of the Tutte polynomial of a matroid is given. Consider the following diagram, where all the maps are natural projections or inclusions:
\begin{center}
	\begin{tikzcd}
		& Fl(1,k,n-1;n) \arrow[ld,"\pi_k"'] \arrow[rd] \arrow[rdd,"\pi_{1(n-1)}"'] &  \\
		G(k,n) & & Fl(1,n-1;n) \arrow[d,hook] \\
		& & \PP^{n-1} \times \PP^{n-1}\simeq G(1,n)\times G(n-1,n)
	\end{tikzcd}
\end{center}
It is known that $K^0(\PP^{n-1} \times \PP^{n-1}) \cong \ZZ[\alpha,\beta]/(\alpha^n,\beta^n)$, where $\alpha$ and $\beta$ are the structure sheaves of hyperplanes.
\begin{theorem}[\kern-0.5em  {\cite[Theorem 7.1]{FinkSpeyer}}] The following equality holds:
	\[
	(\pi_{1(n-1)})_*\pi_d^*(Y(M)\cdot[\mathcal{O}(1)]) = T_M(\alpha,\beta),
	\]
	where $Y(M)$ is the class associated to the matroid $M$ in the non-equivariant K-theory of the Grassmannian.
\end{theorem}
In other words, the Tutte polynomial of a matroid can be viewed as a Fourier-Mukai transform of its associated class in $K$-theory.

We can now now generalize this construction to get a \emph{definition} of the Tutte polynomial of a flag matroid.

\begin{definition}
	Consider the following diagram.
	\begin{center}
		\begin{tikzcd}
			& Fl(1,\bk,n-1;n) \arrow[ld,"\pi_{\bk}"'] \arrow[rd] \arrow[rdd,"\pi_{1(n-1)}"'] &  \\
			Fl(\bk;n) & & Fl(1,n-1;n) \arrow[d,hook] \\
			& & \PP^{n-1} \times \PP^{n-1}
		\end{tikzcd}
	\end{center}
Let $\mathcal{F}$ be a flag matroid on $[n]$ of rank $\bk$, and let $Y(\mathcal{F})\in K^0(Fl(\bk;n))$ be its class in non-equivariant $K$-theory, as in Definition \ref{def:Kclassflag}. Then the \emph{$K$-theoretic Tutte polynomial} of $\mathcal{F}$ is defined to be
\[
T_{\mathcal{F}}(\alpha,\beta):=(\pi_{1(n-1)})_*\pi_d^*(Y(\mathcal{F})\cdot[\mathcal{O}(1)]).
\]
\end{definition}
We computed the Tutte polynomial for some small examples using Sage \cite{sage}, Macaulay2 \cite{M2}, and Normaliz \cite{Normaliz}. Our code is available on \url{software.mis.mpg.de}. The program first computes the equivariant class $(\pi_{1(n-1)})_*\pi_d^*(y(M)\cdot[\mathcal{O}(1)]) \in K^0_T(\PP^{n-1} \times \PP^{n-1})$ using equivariant localisation, and then computes the underlying non-equivariant class.\\
 \begin{example} \label{eg:Tutte1}
 	We consider again the flag matroid of Examples \ref{eg:flagmatroid} and \ref{eg:flagmatroidclass}. We first compute $y(\mathcal{F})\cdot[\mathcal{O}(1)]$, which is displayed in figure \ref{fig:flagTutte}.\\
 		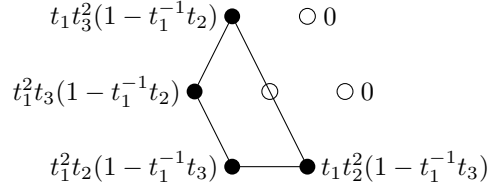
\begin{figure}[h!]
 		\begin{tikzpicture}
 		\coordinate (a) at (0,0);
 		\coordinate (b) at (1,0);
 		\coordinate (c) at (-0.5,1);
 		\coordinate (d) at (0.5,1);
 		\coordinate (e) at (1.5,1);
 		\coordinate (f) at (0,2);
 		\coordinate (g) at (1,2);
 		\node (1) at (-1.3,0) {$t_1^2t_2(1-t_1^{-1}t_3)$};
 		\node (2) at (-1.8,1) {$t_1^2t_3(1-t_1^{-1}t_2)$};
 		\node (3) at (-1.3,2) {$t_1t_3^2(1-t_1^{-1}t_2)$};
 		\node (4) at (2.3,0) {$t_1t_2^2(1-t_1^{-1}t_3)$};
 		\node (5) at (1.8,1) {0};
 		\node (6) at (1.3,2) {0};
 		\draw (a) -- (c) -- (f)--(b)--(a);
 		\foreach \pt in {a,b,c,f} \fill[black] (\pt) circle (3pt);
 		\foreach \pt in {e,d,g}  \draw (\pt) circle (3pt);
 		
 		\end{tikzpicture}
 		\caption{$y(\mathcal{F})\cdot[\mathcal{O}(1)]$} \label{fig:flagTutte}
 	\end{figure}
 The two projections from $Fl(1,1,2,2;3)$ to $Fl(1,2;3)$ are isomorphisms, hence pulling back and pushing forward along them does nothing. Next we need to push our class $X=y(\mathcal{F})\cdot[\mathcal{O}(1)] \in K^0_T(Fl(1,2;3))$ to a class $Y \in K^0_T(\PP^{n-1} \times \PP^{n-1})$, using formula (\ref{eqn:pushforward}).

 The $T$-fixed points of $\PP^{n-1} \times \PP^{n-1}$ are given by pairs $p=(\ell,H)$, where $\ell \in G(1,3)^T =\{\langle e_1\rangle ,\langle e_2\rangle,\langle e_3\rangle \}$ and $H \in G(2,3)^T = \{\langle e_1,e_2 \rangle,\langle e_1,e_3 \rangle, \langle e_2,e_3 \rangle\}$. If $\ell \not\subset H$, then $Y(p)=0$.
 If $\ell \subset H$, then $p \in Fl(1,2;3) \subset \PP^2 \times \PP^2$. Since we are pushing forward along an embedding, the formula (\ref{eqn:pushforward}) has a simple form: we can find characters $\chi_1,\chi_2,\chi_3,\eta$ and open neighbourhoods $p \ni U_1 \subset Fl(1,2;3)$ and $p \ni U_2 \subset \PP^2 \times \PP^2$, such that $T$ acts on $U_1$ with characters $\chi_1,\chi_2,\chi_3$, and on $U_2$ with characters $\chi_1,\chi_2,\chi_3,\eta$. Then (\ref{eqn:pushforward}) becomes:\[
  Y(p) = (1-\eta^{-1})X(p).
 \]

 Consider for example $p=(\langle e_1\rangle,\langle e_1,e_2 \rangle)$. Then $p \in Fl(1,2;3)$ has an open neighbourhood where $T$ acts by characters $t_2t_3^{-1},t_1t_3^{-1},t_1t_2^{-1}$, while $p \in \PP^{2} \times \PP^{2}$ has an open neighbourhood where $T$ acts by characters $t_2t_3^{-1},t_1t_3^{-1},t_1t_3^{-1},t_1t_2^{-1}$. We compute that
 \[
 Y((\langle e_1\rangle,\langle e_1,e_2 \rangle)) =t_1^2t_2(1-t_1^{-1}t_3)(1-t_1^{-1}t_3)=t_2(t_1-t_3)^2.
 \]

 Similarily, we find that
 \begin{align*}
 Y((\langle e_1\rangle,\langle e_1,e_3 \rangle)) &= t_3(t_1-t_2)^2\\
 Y((\langle e_3\rangle,\langle e_1,e_3 \rangle)) &= t_3(t_1-t_2)(t_3-t_2)\\
 Y((\langle e_2\rangle,\langle e_1,e_2 \rangle)) &= t_2(t_1-t_3)(t_2-t_3)
 \end{align*}
 and $Y(p)=0$ in all other cases.

 Finally, we need to find the underlying class in non-equivariant $K$-theory. This is quite tedious to do by hand, so we just refer to the algorithm provided in \url{software.mis.mpg.de} for this.
 In the end, we find that
 \[
 T_{\mathcal{F}}(x,y) = x^2y^2 + x^2y + xy^2 + x^2 + xy.
 \]
 \end{example}
\begin{example} \label{eg:Tutte2}
As another example, consider the uniform flag matroid $\mathcal{U}_{(2,3);5}$ of rank $(2,3)$ on $[5]$ (that is, the constituents of $\mathcal{U}_{2,3}$ are the uniform matroids $U_{2,5}$ and $U_{3,5}$). Using our program, we find that its $K$-theoretic Tutte polynomial $T_{\mathcal{U}_{(2,3):5}}(x,y)$ equals
\[
x^3y^3 + 2x^3y^2 + 2x^2y^3 + 3x^3y + 8x^2y^2 + 3xy^3 + 4x^3 + 8x^2y + 8xy^2 + 4y^3 + 2x^2 + 4xy + 2y^2.
\]
\end{example}

\section{Open Problems}
Our definition of the Tutte polynomial of a flag matroid is admittedly quite involved. It is natural to wonder whether there is an easier definition, avoiding geometry:
\begin{problem}
	Is there a purely combinatorial description of the $K$-theoretic Tutte polynomial of a flag matroid? In particular, can one obtain the $K$-theoretic Tutte polynomial from the Tutte polynomials of the constituents?
\end{problem}
We can also ask which properties of (the coefficients of) the Tutte polynomial for matroids carry over to polymatroids. For example:
\begin{conjecture}
	The coefficients of the $K$-theoretic Tutte polynomial of a flag matroid are always nonnegative integers.
\end{conjecture}
For matroids, one can define the \emph{characteristic polynomial} (also called \emph{chromatic polynomial}, as it generalises the chromatic polynomial of a graph) by
\[
\chi_M(\lambda) = (-1)^{r(M)} T_M(1-\lambda,0).
\]
In 2015, Adiprasito, Huh and Katz proved the following conjecture by Rota-Heron-Welsh:
 \begin{theorem}[\kern-0.5em {\cite{AdiprasitoHuhKatz} }] \label{thm:AHK}
 	Let $w_i(M)$ be the absolute value of the coefficient of $\lambda^{r(M)-i}$ in the characteristic polynomial of $M$. Then the sequence $w_i(M)$ is log-concave:
 	\[
 	w_{i-1}(M)w_{i+1}(M) \leq w_i(M)^2 \text{ for all } 1 \leq i < r(M).
 	\]
 \end{theorem}
Since we now have a definition for the Tutte polynomial of a flag matroid, we can define the \emph{characteristic polynomial} of a rank $\bk$ flag matroid $\mathcal{F}$ by
\[
\chi_{\mathcal{F}}(\lambda) = (-1)^{r(\mathcal{F})} T_{\mathcal{F}}(1-\lambda,0),
\]
where $r(\mathcal{F}) := |\bk| := \sum{k_i}$.
\begin{conjecture} \label{conj:AHK}
	Theorem \ref{thm:AHK} holds for the characteristic polynomial of an arbitrary flag matroid.
\end{conjecture}
In Examples \ref{eg:Tutte1} and \ref{eg:Tutte2}, the characteristic polynomials are $-\lambda^2 + 2\lambda - 1$ and $4\lambda^3 - 14\lambda^2 + 16\lambda - 6$, respectively. Thus, we see that Conjecture \ref{conj:AHK} holds for these examples.

Flag matroids are a special class of Coxeter matroids. Hence, another possible direction of research would be:
\begin{problem}
	Explore how our constructions and results could be generalised to arbitrary Coxeter matroids.
\end{problem}
Flag matroids can also be viewed as a special class of polymatroids. In particular, we can apply the construction of Section \ref{TuttePoly} to them.
\begin{problem}
	Is there a connection between the Tutte polynomial of a polymatroid, as defined by Cameron and Fink, and our construction of the Tutte polynomial of flag matroid? How about the construction from Remark \ref{rmk:polymatroidmatroid}? 
\end{problem}
Next, we note that by Section \ref{sec:repPoly} we know how to associate to a rank $k$ representable $r$-polymatroid on $[n]$ a class in $K^0_T(G(k,rn))$.
\begin{problem}
	Can we associate a class in $K^0_T(G(k,rn))$ to a non-representable $r$-polymatroid?
Can this be used to define a $K$-theoretic Tutte polynomial for $r$-polymatroids?
\end{problem}
Finally, we could apply the construction of Section \ref{sec:Ktutte} to \emph{any} subvariety of a Grassmannian (or even a flag variety), not just to torus orbits. It could be interesting to study the properties of this invariant.
\bibliography{FlagMatroids}{}
\bibliographystyle{plain}
\end{document}